\documentclass[11pt,reqno,a4paper]{amsart}
 \usepackage[british]{babel}
 \usepackage{amsmath,amsfonts,mathrsfs,amsthm,amssymb}
 \usepackage{amssymb}
 \usepackage{graphicx}
 \usepackage{framed}
 \usepackage{color}
 \usepackage{hyperref}
 \usepackage{mathabx}
 \usepackage{mathtools}
 \usepackage{tikz}
 \usepackage{Dynkin} 
 \usepackage{enumitem}
 \usepackage{DT-macros}
  
 \makeatletter
 \newcommand{\mylabel}[2]{#2\def\@currentlabel{#2}\label{#1}}
 \makeatother

 \renewcommand\GL{\operatorname{GL}}
 \renewcommand\SL{\operatorname{SL}}
 \renewcommand\SO{\operatorname{SO}}

 \newcommand\Stab{\operatorname{Stab}}
 \newcommand\PGL{\operatorname{PGL}}

\newtheorem{theorem}{Theorem}[section]
\newtheorem{lemma}[theorem]{Lemma}
\newtheorem{cor}[theorem]{Corollary}
\newtheorem{prop}[theorem]{Proposition}
\theoremstyle{defn}
\newtheorem{defn}[theorem]{Definition}
\newtheorem{example}[theorem]{Example}

\theoremstyle{remark}
\newtheorem{remark}[theorem]{Remark}
\newtheorem{problem}[theorem]{Problem}
\numberwithin{equation}{section}

\title[On uniqueness of submaximally symmetric parabolic geometries]
      {On uniqueness of submaximally symmetric\\ parabolic geometries}

\author[D.~The]{Dennis~The}

\date{\today}

\address{Department of Mathematics \& Statistics, UiT The Arctic University of Norway, N-9037, Troms\o, Norway}

\email{dennis.the@uit.no}



\begin{document}
 \maketitle
 \thispagestyle{empty}

 \begin{abstract} 
 Among (regular, normal) parabolic geometries of type $(G,P)$, there is a locally unique maximally symmetric structure and it has symmetry dimension $\dim(G)$.  The symmetry gap problem concerns the determination of the next realizable (submaximal) symmetry dimension.  When $G$ is a complex or split-real simple Lie group of rank at least three or when $(G,P) = (G_2,P_2)$, we establish a local uniqueness result for submaximally symmetric structures of type $(G,P)$.
 \end{abstract}

 \section{Introduction}
 \label{S:intro}
 
 For a given (local) differential geometric structure, our interest here will be on the dimension of its Lie algebra of infinitesimal symmetries.  Many types of structures (e.g.\ Riemannian metrics on manifolds of fixed dimension) admit a finite maximal symmetry dimension $\fM$, and there is broad interest to (locally) classify all such maximally symmetric structures.  Letting $\fS$ denote the next possible realizable ({\sl submaximal}) symmetry dimension, there is often a significant gap arising between $\fM$ and $\fS$.  The {\sl symmetry gap problem} refers to the determination of $\fS$ and in doing so the task of exhibiting (local) models realizing this submaximal symmetry dimension.  With this goal in mind, one can make a detailed case-by-case study of the PDE determining the symmetry vector fields for a given structure, but in many situations such a direct investigation using analytic tools becomes cumbersome.  Our approach here is to draw upon strong algebraic tools that are present for an important broad class of structures that admit an equivalent reformulation as Cartan geometries.
 
 Parabolic geometries \cite{CS2009} admit such a reformulation -- they are a diverse and interesting class of geometries whose underlying structures include conformal, projective, CR, 2nd order ODE systems, and many classes of generic distributions, e.g.\ $(2,3,5)$-distributions.  Their description as parabolic geometries (see \S \ref{S:parabolic}) gives a solution to the equivalence problem for such structures in the sense of \'Elie Cartan.  Briefly, such a geometric structure on $M$ (henceforth, always assumed {\em connected}) admits a categorically equivalent description as a (regular, normal) Cartan geometry $(\cG \to M, \omega)$ of type $(G,P)$, where $G$ is a semisimple Lie group and $P$ is a parabolic subgroup.  (For more details on the passage from $M$ to the ``upstairs'' Cartan perspective, we refer the reader to \cite{CS2009,Cap2017}.)  The Cartan connection $\omega$ provides a canonical coframing on $\cG$ and its symmetry algebra $\inf(\cG,\omega)$ is isomorphic to the symmetry algebra of the underlying structure on $M$.  We have $\fM = \dim(G)$ for such structures, and there is a (locally) unique maximally symmetric model, namely the {\sl flat model} $(G\to G/P, \omega_G)$ of type $(G,P)$, where $\omega_G$ is the Maurer--Cartan form of $G$.  Any Cartan geometry of type $(G,P)$ can be viewed as a curved version of this flat model, and our starting point is to take the (normalized) Cartan geometry as the basic input to the problem.

 Substantial progress was made on the symmetry gap problem for parabolic geometries in \cite{KT2017}.  In that joint work with Kruglikov, we proved that $\fS \leq \fU$ for any $(G,P)$ in terms of a universal (algebraically-defined) upper bound $\fU$.  Moreover, when $G$ is {\em complex} or {\em split-real} simple:
 \begin{enumerate}
 \item[(i)] $\fU$ can be efficiently calculated via Dynkin diagram combinatorics, and 
 \item[(ii)] $\fS = \fU$ almost always, with some exceptions when $\rnk(G) = 2$.
 \end{enumerate}
 We uniformly proved $\fS = \fU$ by exhibiting a particular {\em homogeneous} structure, encoded ``{\sl Cartan-theoretically}'' by what we refer to here as an {\sl algebraic model $(\ff;\fg,\fp)$} (see \S \ref{S:alg-models}).  We remark that for more general real forms, the determination of $\fU$ and sharpness of $\fS \leq \fU$ is still largely open, although numerous interesting cases have been resolved -- see for example \cite{DT2014, Kru2016, KMT2016, KWZ2018}.

 Not addressed in \cite{KT2017} was the broader classification problem for submaximally symmetric structures, and our goal in this article is to resolve this.   In order to formulate our main result, we briefly recall some notions here.  (Precise definitions will be given later.)  For any (regular, normal) parabolic geometry, there is a fundamental quantity called {\sl harmonic curvature} $\kappa_H : \cG \to H_2(\fp_+,\fg)^1$, which completely obstructs local equivalence to the flat model.  The codomain of $\kappa_H$ is a filtrand of a certain Lie algebra homology group, which is a completely reducible $P$-representation, so only the action on it by the (reductive) Levi factor $G_0 \subsetneq P$ is relevant.  Consider a $G_0$-irrep $\bbV \subseteq H_2(\fp_+,\fg)^1$. We say that $(\cG \to M, \omega)$ is of type $(G,P,\bbV)$ if it is of type $(G,P)$ and $\im(\kappa_H) \subseteq \bbV$, and let $\fS_\bbV$ be the maximal symmetry dimension among regular, normal parabolic geometries of type $(G,P,\bbV)$ with $\kappa_H \not\equiv 0$.  We can now formulate our main result\footnote{See \S \ref{S:can-submax} for our subscript notation for parabolics in the complex or split-real setting.}:
 
 \begin{theorem} \label{T:main}
 Let $G$ be a complex or split-real {\em simple} Lie group, $P \subsetneq G$ a parabolic subgroup, and $G_0$ its Levi factor.  Let $(\cG \to M, \omega)$ be a regular, normal parabolic geometry of type $(G,P,\bbV)$, where $\bbV \subseteq H_2(\fp_+,\fg)^1$ is a $G_0$-irrep.  Suppose that $\dim(\inf(\cG,\omega)) = \fS_\bbV$, and $\rnk(G) \geq 3$ or $(G,P) = (G_2,P_2)$.  Then the geometry is locally homogeneous about any $u \in \cG$ with $\kappa_H(u) \neq 0$.  The corresponding algebraic model $(\ff;\fg,\fp)$ with $\dim\, \ff = \fS_\bbV$ is (up to $P$-equivalences $\ff \mapsto \Ad_p \ff$, $\forall p \in P$):
 \begin{enumerate}
 \item complex case: unique.
 \item split-real case: one of at most two possibilities.  Uniqueness holds if and only if there exists $g_0 \in G_0$ such that $g_0 \cdot \phi_0 = -\phi_0$, where $\phi_0 \in \bbV$ is a lowest weight vector.
 \end{enumerate}
 \end{theorem}
 
 Our result is {\em constructive} (see \S\ref{S:gen-pf}): over $\bbC$, the distinguished algebraic model $(\ff;\fg,\fp)$ encoding the corresponding submaximally symmetric geometry is what we refer to here as the {\sl canonical curved model of type $(\fg,\fp,\bbV)$}, which has curvature $\kappa = \phi_0$ (interpreted as a harmonic 2-cochain).  The Lie algebra $\ff$ arises as a {\sl filtered deformation} of a graded subalgebra $\fa := \fa^{\phi_0} \subseteq \fg$ (see \S\ref{S:sym-pr}), namely $\ff = \fa$ as vector subspaces, but with bracket $[\cdot,\cdot]_\ff := [\cdot,\cdot] - \kappa(\cdot,\cdot)$, where $[\cdot,\cdot]$ is the bracket on $\fg$ (restricted to $\fa$).  This is the same abstract model used in \cite{KT2017}.  In the split-real setting, the second possibility is $\ff = \fa$ with $\kappa = -\phi_0$.
 
 For fixed $(G,P)$, Theorem \ref{T:main} can be used to deduce the analogous classification of all submaximally symmetric structures, i.e.\ $\kappa_H$ is not constrained to a specific $\bbV$.  See \S \ref{S:examples} for some examples.
 
 We now give numerous examples illustrating that one cannot in general weaken the hypotheses of Theorem \ref{T:main} and expect such a uniform conclusion.
 
 Non-uniqueness over $\bbC$ can occur if we do not require $\bbV \subseteq H_2(\fp_+,\fg)^1$ to be $G_0$-irreducible:
 
 \begin{example} \label{X:LC}
 A Legendrian contact geometry (over $\bbF = \bbR$ or $\bbC$) is a contact manifold $(M^{2n+1},\cC)$ with contact distribution endowed with a splitting $\cC = \cE \op \cF$ into Legendrian subbundles.  (Second order ODE is the $n=1$ case.)  Such a structure underlies a parabolic geometry of type $(\SL(n+2,\bbF),P_{1,n+1})$, $\fg_0 \cong \bbF^2 \times \fsl(n,\bbF)$, and for $n \geq 2$ we have an $\fg_0$-irreducible decomposition $H_2(\fp_+,\fg)^1 \cong \bbT_1 \op \bbT_2 \op \bbW$.
From \cite[Table 11]{KT2017}, we have $\fS_{\bbT_1} = \fS_{\bbT_2} = \fS_\bbW = n^2 + 4$.  The corresponding canonical curved models are inequivalent.
 \end{example}
 
 If one weakens the complex / split-real assumptions, varying phenomena can occur:
 
 \begin{example}
 Real hypersurfaces in $\bbC^3$ having {\bf positive-definite} Levi form yield 5D (integrable) CR geometries, which are specific real forms of complex Legendrian contact geometries (Example \ref{X:LC}) when $n=2$.  They underlie regular, normal parabolic geometries of type $(G,P_{1,3})$, where $G = \operatorname{SU}(1,3)$ (not split-real), and the complexification of $\kappa_H$ would take values only in $\bbW \otimes_\bbR \bbC$.  We have $\fM = 15$, while it is known that $\fS = \fU = 7$, with infinitely many inequivalent submaximally symmetric models; see \cite[Table 8 (D.7)]{DMT2021}.  In the Levi-indefinite case, $G = \operatorname{SU}(2,2)$ (again, not split-real), $\fM = 15$, and there is a unique local model realizing $\fS = \fU = 8$; see \cite[Table 7 (N.8)]{DMT2021}.
 \end{example}
 
 Now suppose $\rnk(G) = 2$.  In contrast to local uniqueness in the $(G_2,P_2)$ case (both over $\bbC$ and $\bbR$, see \S\ref{S:G2}), there is a 1-parameter family of submaximally symmetric models in the $(G_2,P_1)$ case:
 
 \begin{example} \label{E:235} A $(2,3,5)$-geometry is a 5-manifold $M$ equipped with a rank 2 distribution $\cD$ having generic growth under the Lie bracket, i.e.\ $\rnk([\cD,\cD]) = 3$ and $\rnk([\cD,[\cD,\cD]]) = 5$.  Locally, any such $\cD$ admits a Monge normal form: there exist local coordinates $(x,y,p,q,z)$ and a function $f = f(x,y,p,q,z)$ with $f_{qq} \neq 0$ such that $\cD$ is spanned by the vector fields
 \[
 \partial_q, \quad \partial_x + p\partial_y + q\partial_p + f\partial_z.
 \]
 Such a structure underlies a parabolic geometry of type $(G_2,P_1)$, so $\fM = 14$, with $f = q^2$ realizing maximal symmetry.  Here, $\fS = \fU = 7$ in either the complex or real case.  Over $\bbC$, a well-known list of submaximally symmetric models is given by $f = q^m$ (for $m \neq -1, 0, \frac{1}{3}, \frac{2}{3}, 1, 2$) and $f = \log(q)$.
 \end{example}
 
  Other rank two cases include 3-dimensional conformal geometry, i.e.\ type $(B_2, P_1)$, and the contact geometry of scalar 3rd order ODE, i.e.\ type $(B_2, P_{1,2})$.  Submaximally symmetric models are non-unique for both -- in the former case see the classification in \cite{Kru1954}, while in the latter case they are given by $y''' + ky' + y = 0$, where $k$ is constant.  The rank two case of 2nd order ODE exhibits several exceptional phenomena:
 
 \begin{example}\label{X:2ODE} Scalar 2nd order ODE $y'' = f(x,y,y')$ (up to point transformations) underlie $(\SL_3,P_{1,2})$ geometries, for which $\fM = 8$ and $\fS = 3 < \fU = 4$.  Locally, one has a 3-manifold $M$ with coordinates $(x,y,p)$ and split contact distribution $\cC = \cE \op \cF$ on $M$ with
 \begin{align}
 \cE = \langle \partial_x + p\partial_y + f(x,y,p) \partial_p \rangle, \quad \cF = \langle \partial_p \rangle.
 \end{align}
 We have $\dim(G_0) = 2$, and $G_0$ corresponds to arbitrary rescalings along $\cE$ and $\cF$.  We have $H_2(\fp_+,\fg)^1 \cong \bbL_1 \op \bbL_2$, with each $\bbL_i$ being a 1-dim $G_0$-irrep.  The components of $\kappa_H$ along $\bbL_1$ and $\bbL_2$ correspond to the well-known {\sl Tresse relative invariants} $I_1$ and $I_2 = f_{pppp}$.  For $I_1$, we refer to \cite[eqn (5.8)]{KT2017} and replace $(t,x,p)$ there with $(x,y,p)$.  Two submaximally symmetric models are:
 \begin{enumerate}
 \item[(i)] $y'' = \exp(y')$: symmetries are $\ff = \langle \partial_x, \partial_y, x\partial_x + (y-x)\partial_y - \partial_p \rangle$.  We have $I_1 = e^{3p}$ and $I_2 = e^p$ both nonvanishing.  Thus, $\kappa_H$ is not concentrated in a single irreducible component.
 \item[(ii)] $y'' = (xy'-y)^3$: symmetries are $\ff = \langle x\partial_y + \partial_p, x\partial_x - y\partial_y - 2p\partial_p, y\partial_x - p^2 \partial_p \rangle$.
 The evaluation map $\operatorname{ev}_o : \ff \to T_o M$ is surjective except along the singular set $\Sigma = \{ y = px \}$, so neighbourhoods of $o_1 \in \Sigma$ and $o_2 \not\in \Sigma$ (endowed with restricted geometric structures) are not locally equivalent.  We have $I_1 = 72(px-y)$ and $I_2 = 0$, so $\kappa_H$ vanishes along $\Sigma$.
 \end{enumerate}
 \end{example}

 A priori, we cannot exclude the possibility of similar limiting singular behavior as in Example \ref{X:2ODE}(ii) for submaximally symmetric structures occurring in geometries with $\rnk(G) \geq 3$, so we always work near a point where $\kappa_H$ is nonvanishing.  Constraining ourselves to the hypotheses of Theorem \ref{T:main} ultimately leads to a classification problem for {\em homogeneous} structures.

 We note that {\sl Cartan reduction} is a general method for classifying (homogeneous) geometric structures.  (See for example \cite{DMT2020} for a recent application.)  While this is a powerful, systematic method, it is typically applied on a case-by-case basis, and for any given structure it takes a substantial amount of effort to set up the correct structure equations (via the {\sl Cartan equivalence method}, for instance).  Moreover, its implementation can be extremely cumbersome to do by-hand (often being done in a symbolic algebra system such as {\tt Maple} or {\tt Mathematica}), and normalizations generally proceed in an ad-hoc manner.  In principle, it can be used to analyze submaximally symmetric structures, but in practice it is not a feasible method to arrive at the claimed generality of Theorem \ref{T:main}.  Our approach will be to proceed in a {\em uniform} manner by taking the Cartan-geometric viewpoint as the basic input, and make efficient use of representation theory.
  
 Let us briefly outline our article.  In \S\ref{S:ParAlg}, we recall relevant background from parabolic geometries and our earlier work on symmetry gaps, and formulate the notion of an algebraic model $(\ff;\fg,\fp)$ encoding any homogeneous parabolic geometry.  In \S\ref{S:CanCurv}, we recall Kostant's theorem, define the canonical curved model, and formulate the algebraic model classification problem (Problem \ref{P:problem}).  We then solve it, first for $(G_2,P_2)$ geometries (\S \ref{S:G2}), and then the general $\rnk(G) \geq 3$ case (\S \ref{S:gen-pf}).  We conclude in \S\ref{S:examples} with concrete examples of submaximally symmetric structures, which are asserted to be unique (over $\bbC$) from Theorem \ref{T:main}. \\
  
 {\bf Conventions}: The base manifold $M$ is always assumed to be connected.  We work in the smooth and holomorphic categories when referring to real and complex geometries, respectively.  For simple roots, we use the same ordering as in {\tt LiE} \cite{LiE}.

 \section{Parabolic geometries and algebraic models}
 \label{S:ParAlg}
 
 We begin by reviewing background from parabolic geometries and symmetry gaps -- see \cite{CS2009, KT2017} for more details.
 \subsection{Parabolic geometries}
 \label{S:parabolic}
 
 Let $G$ be a real or complex semisimple Lie group, $P \subset G$ a parabolic subgroup, and $\fp \subset \fg$ be their Lie algebras.  Then $\fg$ admits a natural $P$-invariant (decreasing) filtration $\fg = \fg^{-\nu} \supset ... \supset \fg^\nu$ (we put $\fg^i = \fg$ for $i < -\nu$, $\fg^i = 0$ for $i > \nu$), $\fg^1=\fp_+$ is the nilradical of $\fg^0 = \fp$, and $[\fg^i,\fg^j] \subset \fg^{i+j}$ for all $i,j \in \bbZ$.  There always exists {\sl grading element} $\sfZ \in \fg$ whose $\ad_\sfZ$-eigenvalues $\forall j \in \bbZ$ ({\sl degrees}) and eigenspaces $\fg_j := \{ x \in \fg : \ad_\sfZ(x) = j x \}$ ($\forall j \in \bbZ$) endow $\fg$ with the structure of a graded Lie algebra $\fg = \fg_{-\nu} \op \ldots \op \fg_\nu$ compatible with the filtration, i.e.\ $[\fg_i,\fg_j] \subset \fg_{i+j}$ and $\fg^i \cong \bigoplus_{j=i}^\nu \fg_j$. The {\sl associated-graded} Lie algebra $\tgr(\fg)$ is defined by $\tgr_i(\fg) := \fg^i / \fg^{i+1}$.  Given $\sfZ$ as above, we identify $\tgr_i(\fg) \cong \fg_i$ as $\fg_0$-modules, and if $x \in \fg^i$, we denote by $\tgr_i(x) \in \fg_i$ the projection to its {\sl leading part}.  We have $\sfZ \in \fz(\fg_0)$ (centre of $\fg_0$), $\fp = \fg_0 \op \fp_+$,  and the Killing form on $\fg$ identifies $(\fg / \fp)^* \cong \fp_+$ as $P$-modules.  Finally, letting $G_0 = \{ g \in P : \Ad_g (\fg_i) \subset \fg_i, \, \forall i \}$ be the {\sl Levi subgroup} (with Lie algebra $\fg_0$), and $P_+ = \exp(\fp_+) \leq P$, we have $P \cong G_0 \ltimes P_+$.
 
 A {\sl parabolic geometry} is a Cartan geometry $(\cG \to M, \omega)$ of type $(G,P)$, i.e.\ a (right) principal $P$-bundle $\cG \to M$ with a {\sl Cartan connection} $\omega \in \Omega^1(\cG,\fg)$:
 \begin{enumerate}
 \item[(i)] $\omega_u : T_u \cG \to \fg$ is a linear isomorphism $\forall u \in \cG$;
 \item[(ii)] $\omega$ is $P$-equivariant: $R_p^* \omega = \Ad_{p^{-1}} \circ \omega$, $\forall p \in P$;
 \item[(iii)] $\omega(\zeta_A) = A$, $\forall A \in \fp$, where $\zeta_A$ is the fundamental vertical vector field corresponding to $A$.
 \end{enumerate}
  The {\sl curvature} of $\omega$ is $K = d\omega + \frac{1}{2}[\omega,\omega] \in \Omega^2(\cG,\fg)$ (which is $P$-equivariant and {\sl horizontal}, i.e.\ $K(\zeta_A,\cdot) = 0$), or equivalently we have 
the {\sl curvature function} $\kappa : \cG \to \bigwedge^2 (\fg/\fp)^* \otimes \fg$ given by $\kappa(x,y) = K(\omega^{-1}(x), \omega^{-1}(y))$.  The geometry is {\sl flat} if $K=0$, which characterizes local equivalence to the flat model $(G \to G/P, \omega_G)$, where $\omega_G$ is the (left-invariant) Maurer--Cartan form on $G$.  Via the Killing form, the codomain of $\kappa$ identifies (as a $P$-module) with $C_2(\fp_+,\fg):= \bigwedge^2 \fp_+ \otimes \fg$.  These are 2-chains in the complex $(C_\bullet(\fp_+,\fg),\partial^*)$ with $\partial^*$ the Lie algebra homology differential.  We say that $(\cG \to M, \omega)$ is {\sl normal} if $\partial^* \kappa = 0$ and it is {\sl regular} if $\kappa(\fg^i,\fg^j) \subset \fg^{i+j+1}$ for any $i,j$.  Equivalently, if we naturally extend the filtration on $\fg$ to a filtration on $\bigwedge^2 \fp_+ \otimes \fg$, then we have $\kappa \in \ker(\partial^*)^1$.  This is  the subspace of $\ker(\partial^*) \subset \bigwedge^2 \fp_+ \otimes \fg$ on which $\sfZ$ acts with positive eigenvalues ({\sl degrees}).  There is a well-known equivalence of categories between regular, normal parabolic geometries and underlying geometric structures on $M$ (see \cite{CS2009} for details).
 
 For any regular, normal parabolic geometry, a key invariant is its {\sl harmonic curvature} $\kappa_H : \cG \to H_2(\fp_+,\fg) := \frac{\ker(\partial^*)}{\im(\partial^*)}$, given by $\kappa_H = \kappa\,\, \mod \im(\partial^*)$, and this $P$-equivariant function is a complete obstruction to flatness.  Moreover, $H_2(\fp_+,\fg)$ is a completely reducible $\fp$-representation, i.e.\ $\fp_+$-acts trivially.  As $\fg_0$-modules, $\fg / \fp \cong \fg_-$, and $C^k(\fg_-,\fg) := \bigwedge^k \fg_-^* \otimes \fg$ yields a complex $(C^\bullet(\fg_-,\fg),\partial)$ with respect to the standard Lie algebra cohomology differential $\partial$, for which we have the ($\fg_0$-invariant) algebraic Hodge decomposition:
 \begin{align} \label{E:Hodge}
 C^k(\fg_-,\fg) \cong \im(\partial) \op \ker(\Box) \op \im(\partial^*),
 \end{align}
 where $\Box = \partial \partial^* + \partial^* \partial$ is the ($\fg_0$-equivariant) algebraic Laplacian, with $\ker(\Box) = \ker(\partial) \cap \ker(\partial^*)$.   Then $H_2(\fp_+,\fg) = \frac{\ker(\partial^*)}{\im(\partial^*)} \cong \ker(\Box) \cong \frac{\ker(\partial)}{\im(\partial)} \cong H^2(\fg_-,\fg)$ as $\fg_0$-modules, which may be efficiently computed via Kostant's theorem (\S \ref{S:can-submax}).  By regularity, $\kappa_H$ has image in the subspace $H_2(\fp_+,\fg)^1 \subseteq H_2(\fp_+,\fg)$ on which $\sfZ$ acts with positive eigenvalues.  This corresponds to some $\fg_0$-submodule $H^2_+(\fg_-,\fg) \subseteq H^2(\fg_-,\fg)$ under the above identification.
  
 Finally, by \cite[Thm.3.1.12]{CS2009}, if $\kappa$ has lowest non-trivial degree $s > 0$, then its leading part $\tgr_s(\kappa)$ is harmonic and coincides with the degree $s$ component of $\kappa_H \neq 0$.  In particular, $\kappa_H$ being a complete obstruction to flatness follows from this.

 \subsection{Symmetry and Tanaka prolongation}
 \label{S:sym-pr}
 
 Two (regular, normal) parabolic geometries of type $(G,P)$ are {\sl equivalent} if there is a principal bundle isomorphism that pulls back one Cartan connection to the other, and an {\sl automorphism} is a self-equivalence.   A Cartan geometry $(\cG \stackrel{\pi}{\to} M, \omega)$ is {\sl (locally) homogeneous} if there is a Lie group acting by (local) automorphisms whose projection to $M$ yields a (locally) transitive action on $M$.  Infinitesimally, the symmetry algebra is 
 \begin{align}
 \mathfrak{inf}(\cG,\omega) = \{ \xi \in \fX(\cG)^P : \cL_\xi \omega = 0 \},
 \end{align}
 where $\fX(\cG)^P$ are the $P$-invariant vector fields on $\cG$.
 
 Let us now summarize how to equivalently view $\mathfrak{inf}(\cG, \omega)$ in a more algebraic manner  \cite{Cap2005a,CN2009,KT2017}. Fix {\em any} $u \in \cG$.  Then $\omega_u : T_u \cG \to \fg$ restricts to a linear injection on $\inf(\cG,\omega)$.  Letting $\ff = \ff(u) :=\omega_u(\inf(\cG,\omega))$, the Lie bracket on $\inf(\cG,\omega)$ transfers to the bracket on $\ff$ given by
 \begin{align} \label{E:f-br}
 [x,y]_\ff = [x,y] - \kappa_u(x,y), \quad \forall x,y \in \ff.
 \end{align}
 The $P$-invariant filtration on $\fg$ induces a filtration on $\ff$ via $\ff^i := \ff \cap \fg^i$. By regularity, $\kappa(\fg^i,\fg^j) \subset \fg^{i+j+1}$, so $[\ff^i,\ff^j]_\ff \subset \ff \cap \fg^{i+j} = \ff^{i+j}$, and $(\ff,[\cdot,\cdot]_\ff)$ becomes a filtered Lie algebra (generally {\em not} a Lie subalgebra of $\fg$).  By regularity, the {\sl associated-graded} $\fs := \tgr(\ff)$, defined by $\fs_i := \ff^i / \ff^{i+1}$ is identified as a graded subalgebra of $\fg$ (via $\fs_i \inj \ff^i / \fg^{i+1} \subseteq \fg^i / \fg^{i+1} \cong \fg_i$).  The filtrand $\ff^0 \subseteq \fg^0 = \fp$ satisfies the important algebraic condition $\ff^0 \cdot \kappa = 0$, which implies $\ff^0 \cdot \kappa_H = 0$.  Since $\fp_+$ acts trivially on $H_2(\fp_+,\fg)$, then $\ff^1 \cdot \kappa_H = 0$ always, so $\fs_0 \cdot \kappa_H = 0$, i.e. $\fs_0$ is contained in the annihilator $\fa_0 := \fann(\kappa_H(u)) \subseteq \fg_0$.

 Now define the following (extrinsic) {\sl Tanaka prolongation} algebra $\fa^\phi$ as in \cite{KT2017}:
  
 \begin{defn}[Extrinsic Tanaka prolongation] \label{D:TP} Let $\fa_0 \subseteq \fg_0$ be a Lie subalgebra.  Extend this to a $\bbZ$-graded Lie subalgebra $\fa \subseteq \fg$ by defining $\fa_- = \fg_-$ and $\fa_k = \{ X \in \fg_k : [X,\fg_{-1}] \subseteq \fa_{k-1} \}$ for $k > 0$.  Denote $\fa = \bop_k \fa_k$ by $\tpr^\fg(\fg_-,\fa_0)$.  When $\phi$ lies in some $\fg_0$-representation, we write $\fa^\phi := \tpr^\fg(\fg_-,\fann(\phi))$.
 \end{defn}
 
 The constraint $\fs_0 \subseteq \fa_0$ propagates via Tanaka prolongation to the higher levels.
 More precisely, the following important inclusion holds:
 \begin{align} \label{E:s-Kh}
 \fs(u) \subseteq \fa^{\kappa_H(u)}, \quad \forall u \in \cG.
 \end{align}
 Otherwise put, the symmetry algebra $\ff$ is a {\sl constrained filtered sub-deformation of $\fa^{\kappa_H}$}, i.e.\
 \begin{enumerate}
 \item[(i)] $\ff$ is a filtered deformation of the graded subspace $\fs(u) \subseteq \fa^{\kappa_H(u)}$, and
 \item[(ii)] $\ff$ is constrained: e.g.\ it is a filtered subspace of $\fg$ and satisfies \eqref{E:f-br} for some $\kappa$.
 \end{enumerate}
 The inclusion \eqref{E:s-Kh} was established in \cite[Thm.2.4.6]{KT2017} on the open dense set of so-called {\sl regular points}, i.e.\ those $u \in \cG$ on which $\dim\, \fs_i(u)$ are locally constant functions $\forall i$, and was generalized to {\em all} points in  \cite[Thm.3.3]{KT2018}.  If the given geometry is not flat, then $\kappa_H(u) \neq 0$ at some $u \in \cG$, so $\sfZ \not\in \fann(\kappa_H(u))$ by the regularity assumption on $\kappa$, and hence $\dim(\inf(\cG,\omega)) = \dim(\fs) < \dim(\fg)$.  Thus, the flat model is locally the unique maximally symmetric geometry.  Defining
 \begin{align} 
 \fS &:= \max\{ \dim(\inf(\cG,\omega)) : (\cG \to M,\omega) \mbox{ regular, normal of type $(G,P)$ and } \kappa_H \not\equiv 0 \}, \label{E:fS}\\
 \fU &:= \max\{ \dim(\fa^\phi) : 0 \neq \phi \in H^2_+(\fg_-,\fg) \}, \label{E:fU}
 \end{align}
 equation \eqref{E:s-Kh} immediately implies
 \begin{align}
 \fS \leq \fU < \dim(\fg).
 \end{align}
 A (regular, normal) geometry with $\dim(\inf(\cG,\omega)) = \fS$ is {\sl submaximally symmetric}.  A priori, it should not be assumed that these are locally homogeneous, particularly if $\fS < \fU$.  

 \begin{defn}
 Let $\cO \subseteq H_2(\fp_+,\fg)^1$ be a $G_0$-invariant subset.  Let 
 \begin{align*}
 \fS_\cO &:= \max\{ \dim(\inf(\cG,\omega)) : (\cG \to M,\omega) \mbox{ regular, normal of type $(G,P)$,  } \im(\kappa_H) \subseteq \cO,\, \kappa_H \not\equiv 0 \},\\
 \fU_\cO &:= \max\{ \dim(\fa^\phi) : 0 \neq \phi \in \cO \}.
 \end{align*}
 \end{defn}

 \begin{lemma} \label{L:hom}
 For regular, normal parabolic geometries of type $(G,P)$, and $\cO \subseteq H_2(\fp_+,\fg)^1$ a $G_0$-invariant subset, suppose that $\fS_\cO = \fU_\cO$.  Then any $(\cG \to M, \omega)$ with $\im(\kappa_H) \subseteq \cO$ and $\dim(\inf(\cG,\omega)) = \fU_\cO$ is locally homogeneous near any $u \in \cG$ with $\kappa_H(u) \neq 0$.
 \end{lemma}
 
 \begin{proof}
 Fix $u \in \cG$ with $\kappa_H(u) \neq 0$.  By \eqref{E:s-Kh}, $\dim(\fs(u)) = \dim(\fa^{\kappa_H(u)}) \leq \fU_\cO$.  By hypothesis, $\dim(\fs(u)) = \dim(\inf(\cG,\omega)) = \fU_\cO$, so \eqref{E:s-Kh} implies $\fs(u) = \fa^{\kappa_H(u)}$.  Hence, $\fs_-(u) = \fg_-$, which implies local homogeneity by Lie's third theorem.
 \end{proof}

 \subsection{Homogeneous parabolic geometries}
 
 Let $(\cG \to M, \omega)$ be homogeneous with respect to the Lie group $F$.  Fix $u \in \cG$, and let $F^0 \subset F$ be the stabilizer of $o = \pi(u) \in M$.  Given any $f_0 \in F^0$, we have $f_0 \cdot u = u\cdot \iota(f_0)$ for some Lie group homomorphism $\iota : F^0 \to P$.  This defines a right $F^0$-action on $F\times P$ via $(f,p)\cdot f_0 = (ff_0,\iota(f_0^{-1})p)$ and we let $F\times_{F^0} P$ be the collection of all $F^0$-orbits $\overline{(f,p)}$.  Letting $\ff$ and $\ff^0$ be the Lie algebras of $F$ and $F^0$ respectively, we have \cite[Prop.1.5.15]{CS2009}:

 \begin{enumerate}
 \item $\cG\to M$ is equivalent to the associated bundle $F \times_{F^0} P \to  F/F^0$.
 \item Any $F$-invariant Cartan connection $\omega \in \Omega^1(F \times_{F^0} P, \fg)$ of type $(\fg,P)$ is completely determined by the following:
 \end{enumerate} 

 \begin{defn} \label{D:ACC} 
 An {\sl algebraic Cartan connection of type $(\fg,P)$ on $(\ff,F^0)$} is a linear map $\varpi: \ff \to \fg$ with:
 \begin{enumerate}
 \item[(C1)] $\varpi|_{\ff^0} = \iota'$, where $\iota' : \ff^0 \to\fp$ is the differential of $\iota : F^0 \to P$.
 \item[(C2)] $\Ad_{\iota(f)} \circ \varpi = \varpi \circ \Ad_f$, $\forall f \in F^0$.  Infinitesimally:
 \begin{align}
  [\varpi(x),\varpi(y)] = \varpi([x,y]_\ff), \quad \forall x \in \ff^0, \quad \forall y \in \ff, \tag{C2'}\label{E:C2p}
 \end{align}
 where $[\cdot,\cdot]_\ff$ and $[\cdot,\cdot]$ are the Lie brackets on $\ff$ and $\fg$ respectively.  If $F^0$ is connected, then $({\rm C}2)$ and \eqref{E:C2p} are equivalent.
 \item[(C3)] $\varpi$ induces a vector space isomorphism $\ff/ \ff^0 \cong \fg / \fp$.
 \end{enumerate}
 \end{defn}

 Indeed, given $\varpi$ as above, we obtain $\omega$ by factoring $\hat\omega^1 \in \Omega^1(F \times P, \fg)$ given by
 \begin{align*}
 \hat\omega_{(f,p)}(X,Y) = \Ad_{p^{-1}} \varpi(X) + Y, \quad (X,Y) \in T_f F \times T_p P.
 \end{align*}
 The basepoint change $u \mapsto f\cdot u$ leaves $(\iota,\varpi)$ unchanged, but a fibrewise change $u \mapsto u \cdot p$ induces $(\iota,\varpi) \mapsto (\Ad_{p^{-1}} \circ \iota, \Ad_{p^{-1}} \circ \varpi)$.
  
 Define $\tilde\kappa(x,y) := [\varpi(x),\varpi(y)] - \varpi([x,y]_\ff)$, so $\tilde\kappa \in \bigwedge^2(\ff/\ff^0)^* \otimes \fg$ by \eqref{E:C2p}.  The curvature  of $\omega$ corresponds to $\kappa \in \bigwedge^2(\fg/\fp)^*\otimes \fg$ given by $\kappa(x,y) = \tilde\kappa(\varpi^{-1}(x),\varpi^{-1}(y))$.  The notions of regularity and normality of $\kappa$ are immediately specialized to this algebraic setting, as is the quotient object $\kappa_H = \kappa \,\mod \im(\partial^*) \in H_2(\fp_+,\fg)^1$.
  
 \subsection{Algebraic models}
 \label{S:alg-models}
 
 Note that $({\rm C}3)$ and \eqref{E:C2p} forces $\ker(\varpi) \subset \ff^0$ to be an ideal in $\ff$.  The $F$-action on $F/F^0$ can always assumed to be {\sl infinitesimally effective}, i.e.\ $\ff^0$ does not contain any non-trivial ideals of $\ff$ (hence, $\ker(\varpi) = 0$).  (Otherwise, we may without loss of generality quotient both $F$ and $F^0$ by the corresponding normal subgroup.)
 Consequently, we assume that $\varpi : \ff \to \fg$ is injective and identify $\ff$ with its image in $\fg$. This motivates the following definition.
 
 \begin{defn} \label{D:alg-model}
 An {\em algebraic model $(\ff;\fg,\fp)$} is a Lie algebra $(\ff,[\cdot,\cdot]_\ff)$ such that:
 \begin{enumerate}[label=\textnormal{(\arabic*)}]
 \item[\mylabel{D:M1}{(M1)}] $\ff \subseteq \fg$ is a vector subspace with inherited filtration $\ff^i := \ff \cap \fg^i$ such that $\fs = \tgr(\ff)$ satisfies $\fs_- = \fg_-$.
 \item[\mylabel{D:M2}{(M2)}] $\ff^0$ inserts trivially into $\tilde\kappa(x,y) = [x,y] - [x,y]_\ff$, so identify $\tilde\kappa \in \bigwedge^2(\ff/\ff^0)^* \otimes \fg$ with $\kappa \in \bigwedge^2 (\fg/\fp)^* \otimes \fg$.
 \item[\mylabel{D:M3}{(M3)}] $\kappa$ is regular and normal, i.e.\ $\kappa \in \ker(\partial^*)^1$.
 \end{enumerate}
 \end{defn} 
 
 The result below immediately follows from the general theory recalled in \S \ref{S:sym-pr}, but it is instructive to give proofs directly following from Definition \ref{D:alg-model} above.
 
 \begin{prop} \label{P:algCC} 
 Let $(\ff;\fg,\fp)$ be an algebraic model.  Then
 \begin{enumerate}
 \item $(\ff,[\cdot,\cdot]_\ff)$ is a filtered Lie algebra.  (In general, $\ff$ is not a Lie subalgebra of $\fg$.)
 \item $\ff^0 \cdot \kappa = 0$, i.e.\ $[z,\kappa(x,y)] - \kappa([z,x],y) - \kappa(x,[z,y]) = 0$, $\forall x,y \in \ff$ and $\forall z \in \ff^0$.
 \item $\fs \subseteq \fa^{\kappa_H}$, where $\kappa_H := \kappa \,\mod \im(\partial^*)$.
 \end{enumerate}
 \end{prop}
 
 \begin{proof}
 For (1), if $x \in \ff^i$ and $y \in \ff^j$, then $\kappa(x,y) \in \fg^{i+j+1}$ by regularity of $\kappa$, i.e.\ \ref{D:M3}, so (1) implies $[x,y]_\ff \in \fg^{i+j}$.  But $\ff$ is a Lie algebra, so $[x,y]_\ff \in \ff \cap \fg^{i+j} = \ff^{i+j}$ and $[\ff^i,\ff^j]_\ff \subset \ff^{i+j}$.  For (2), we use the Jacobi identity and the fact that $\kappa$ vanishes under $\ff^0$-insertions by \ref{D:M2}.  Namely, let $x,y \in \ff$ and $z \in \ff^0$.  Then
 \begin{align*}
 0 &= [[x,y]_\ff,z]_\ff + [[y,z]_\ff,x]_\ff + [[z,x]_\ff,y]_\ff \\
 &= [[x,y] - \kappa(x,y),z]_\ff +  [[y,z],x]_\ff + [[z,x],y]_\ff \\
 &= [[x,y],z] - [\kappa(x,y),z] +  [[y,z],x] - \kappa([y,z],x) + [[z,x],y] - \kappa([z,x],y) \\
 &= [z,\kappa(x,y)] - \kappa([y,z],x) - \kappa([z,x],y).
 \end{align*}
 
 Finally, we prove (3).  Since $\partial^*$ is $P$-equivariant and $\ff^0 \subseteq \fp = \fg^0$, then $\ff^0 \cdot \im(\partial^*) \subseteq \im(\partial^*)$, so (2) implies $\ff^0 \cdot \kappa_H = 0$, which factors to $\fs_0 \cdot \kappa_H = 0$ by complete reducibility of $H_2(\fp_+,\fg)$.  Letting $\fa := \fa^{\kappa_H}$, this means $\fs_0 \subseteq \fa_0 = \fann(\kappa_H)$.   By regularity, $\fs \subseteq \fg$ is a graded Lie subalgebra, so for any $k > 0$, $[\fs_k,\fg_{-1}] = [\fs_k,\fs_{-1}] \subseteq \fs_{k-1}$.  Inductively, we have $\fs_k \subseteq \fa_k$ for all $k > 0$.
 \end{proof}

 Importantly, we note that the set of algebraic models of type $(\fg,\fp)$:
 \begin{itemize}
 \item admits a $P$-action via $\ff \mapsto \Ad_p \ff$ for any $p \in P$.  All algebraic models in the same $P$-orbit are to be regarded as equivalent, so we must always account for this redundancy.
 \item is partially ordered: declare that $\ff \leq \ff'$ if and only if $\ff$ is a Lie subalgebra of $\ff'$.  We will focus on {\sl maximal elements} $\ff$.  (We view non-maximal elements as non-optimal descriptions of the same geometric structure.)
\end{itemize}

 \begin{remark}
Conversely, by \cite[Lemma 4.1.4]{KT2017}, to each algebraic model $(\ff;\fg,\fp)$, there exists a locally homogeneous geometry $(\cG \to M, \omega)$ of type $(G,P)$ with $\inf(\cG,\omega)$ containing a subalgebra isomorphic to $\ff$.  If $\ff$ is maximal with respect to the partial order defined above, then $\inf(\cG,\omega) \cong \ff$.
 \end{remark}

 Since $\tgr(\ff) = \fs$, we may view $\ff \subseteq \fg$ as a graph over $\fs \subseteq \fg$.  Namely, choosing some graded subspace $\fs^\perp \subseteq \fg$ so that $\fg = \fs \op \fs^\perp$ (in fact, $\fs^\perp \subseteq \fp$ since $\fg_- \subseteq \fs$ by hypothesis), we can write 
 \begin{align}
 \ff = \bigoplus_{i=-\nu}^\nu \{ x + \fd(x) : x \in \fs_i \}
 \end{align}
 for some {\em unique} linear map $\fd : \fs \to \fs^\perp$ of {\em positive} degree, i.e.\ if $x \in \fs_i$, then $\fd(x) \in \fs^\perp \cap \fg^{i+1}$.  We refer to $\fd$ as the {\sl deformation map} and $\fd(x)$ as the {\sl tail} of $x$.
 
 \begin{lemma} \label{L:def-T-inv}
 Let $T \in \ff^0$ with $\fs$ and $\fs^\perp$ being $\ad_T$-invariant.  Then $T \cdot \fd = 0$, i.e.\ $\ad_T \circ \fd = \fd \circ \ad_T$.
 \end{lemma}
 
 \begin{proof} Given $x \in \fs$, we have $x + \fd(x) \in \ff$ and
 $[T,x + \fd(x)]_\ff \in \ff$.  By \ref{D:M2}, $[T,x + \fd(x)]_\ff= [T,x + \fd(x)] = [T,x] + [T,\fd(x)]$.  But $[T,x] \in \fs$ and $[T,\fd(x)] \in \fs^\perp \cap \fg^{i+1}$ since $\fd$ has positive degree, so by uniqueness of $\fd$, we have $[T,\fd(x)] = \fd([T,x])$.
 \end{proof}
 
 \section{The canonical curved model and local uniqueness}
 \label{S:CanCurv}
 
 We focus on proving Theorem \ref{T:main} in the complex case.  The arguments in the split-real case are almost exactly the same, and potentially differ only in the final step of \S \ref{S:gen-pf}.  This is described at the end of  \S\ref{S:CanCurv}.

 \subsection{Kostant's theorem and the canonical curved model}
 \label{S:can-submax}
 
 Let $G$ be a complex semisimple Lie group, $\fh \subset \fg$ a Cartan subalgebra with $\ell := \dim(\fh) = \rnk(\fg)$, $\Delta \subset \fh^*$ the associated root system, $\fg_\alpha$ the root space for $\alpha \in \Delta$, $\Delta^+ \subset \Delta$ the positive roots relative to a choice of simple root system $\{ \alpha_i \}_{i=1}^\ell$, and 
$\{ \sfZ_i \}_{i=1}^\ell \subset \fh$ its dual basis, i.e.\ $\sfZ_i(\alpha_j) = \alpha_j(\sfZ_i) = \delta_{ij}$.  If $\fk \subseteq \fg$ is an $\fh$-invariant subspace, we write $\Delta(\fk) := \{ \alpha \in \Delta : \fg_\alpha \subseteq \fk \}$, and $\Delta^+(\fk) := \Delta(\fk) \cap \Delta^+$.  A parabolic subgroup $P \subset G$ with Lie algebra $\fp \subset \fg$ is encoded by a subset $I_\fp \subseteq \{ 1, \ldots, \ell \}$, with associated grading element $\sfZ  := \sum_{i\in I_\fp} \sfZ_i$, and $\fp = \fg_{\geq 0} = \bigoplus_{i \geq 0} \fg_i$ relative to it.  On the Dynkin diagram of $\fg$, we put crosses at nodes corresponding to $I_\fp$, and refer to $P$ using subscripts, e.g.\ if $I_\fp = \{ i,j \}$, then $P = P_{i,j}$, etc.  (In our convention, the Borel subalgebra has crosses on {\em all} Dynkin diagram nodes.)

 The Killing form of $\fg$ induces a non-degenerate pairing $\langle \cdot, \cdot \rangle$ on $\fh^*$.  Letting $\alpha^\vee := \frac{2\alpha}{\langle \alpha, \alpha \rangle}$ be the coroot of $\alpha \in \Delta$, we have the Cartan matrix $c = (c_{ij})$ with $c_{ij} := \langle \alpha_i, \alpha_j^\vee \rangle$.  The fundamental weights $\{ \lambda_j \}_{j=1}^\ell$ are defined by $\langle \lambda_i, \alpha_j^\vee \rangle = \delta_{ij}$, and these satisfy $\alpha_i = \sum_{j=1}^\ell c_{ij} \lambda_j$.  Corresponding to $\alpha_j$ is the simple reflection $\sigma_j$ defined by $\sigma_j(x) = x - \langle x, \alpha_j^\vee \rangle \alpha_j$, and the Weyl group $W$ is the group generated by all simple reflections.

 Kostant's theorem \cite{Kos1961} yields an efficient $\fg_0$-module description of $H_2(\fp_+,\fg) \cong H^2(\fg_-,\fg)$: it is the direct sum of $\fg_0$-irreps $\bbV_\mu$, each of multiplicity one, and having {\em lowest} weight $\mu = -w\bullet \lambda$, where:
 \begin{enumerate}
 \item $\lambda$ is the highest weight of a simple ideal of $\fg$;
 \item $w = (jk) := \sigma_j \circ \sigma_k \in W^\fp(2)$ is a length 2 word of the Hasse diagram $W^\fp \subset W$.  Concretely for our purposes here, this is equivalent to: $j \in I_\fp$ and either: (i) $k \in I_\fp$, or (ii) $c_{jk} < 0$.
 \item $\bullet$ refers to the affine action of $W$ on weights: letting $\rho := \sum_{i=1}^\ell \lambda_i$, we have
 \begin{align} \label{E:mu}
 \mu &= -w\bullet \lambda = -w(\lambda+\rho) + \rho = -w\bullet 0 + w(-\lambda).
 \end{align}
 \end{enumerate}
 Via the $\fg_0$-module isomorphism $H^2(\fg_-,\fg) \cong \ker(\Box)$ from \eqref{E:Hodge}, a representative lowest weight vector $\phi_0 \in \bbV_\mu \subset \ker(\Box) \subset \bigwedge^2(\fg_-)^* \otimes \fg \cong \bigwedge^2 \fp_+ \otimes \fg$ is given in terms of root vectors $e_\gamma$ by:
 \begin{align} \label{E:phi0}
 \phi_0 := e_{\alpha_j} \wedge e_{\sigma_j(\alpha_k)} \otimes e_{w(-\lambda)}.
 \end{align}
 To interpret $\phi_0$ as a 2-cochain, we identify $e_{\alpha_j}$ and $e_{\sigma_j(\alpha_k)}$ as the dual elements $(e_{-\alpha_j})^*$ and $(e_{-\sigma_j(\alpha_k)})^*$ via the Killing form.  (Here, we fix root vectors yielding a basis on $\fg/\fp$.)
 
 For $\cO = \bbV_\mu \backslash \{ 0 \}$, define $\fU_\mu := \fU_\cO$ and $\fS_\mu := \fS_\cO$.  By \cite[Prop.3.1.1]{KT2017}, we have 
 \begin{align} \label{E:TPmax}
 \dim(\fa^\phi) \leq \dim(\fa^{\phi_0}), \quad \forall \phi \in \bbV_\mu \backslash \{ 0 \},
 \end{align}
 with equality precisely when the projectivizations $[\phi]$ and $[\phi_0]$ lie in the same $G_0$-orbit in $\bbP(\bbV_\mu)$.  Hence, $\fU_\mu = \dim(\fa^{\phi_0})$.  Concerning realizability of this upper bound, we have:
 
 \begin{defn} \label{D:canonical}
 Use notations as above with $G$ simple.  Suppose $w \in W^\fp(2)$ satisfies $w(-\lambda) \in \Delta^-$ and $\sfZ(\mu) > 0$.  The {\sl canonical curved model of type $(\fg,\fp,\bbV_\mu)$} is the algebraic model $(\ff;\fg,\fp)$ given by
 defining $\ff := \fa^{\phi_0}$ as a vector subspace of $\fg$, equipped with the filtration inherited from $\fg$, and deformed bracket $[\cdot,\cdot]_\ff := [\cdot,\cdot] - \phi_0(\cdot,\cdot)$.
 \end{defn}
 
By \cite[Lemma 4.1.1]{KT2017}, $(\ff,[\cdot,\cdot]_\ff)$ is indeed a Lie algebra, and clearly $\dim(\ff) = \fU_\mu$.  The filtration on $\ff$ is inherited: $\ff^i := \ff \cap \fg^i$. Since $\phi_0 \in \ker(\Box) = \ker(\partial) \cap \ker(\partial^*)$, then $\kappa = \phi_0$ is clearly normal, and $\sfZ(\mu) > 0$ guarantees regularity.  Thus, $(\ff;\fg,\fp)$ is an algebraic model, which is clearly maximal with respect to the aforementioned partial order.
 
 \begin{prop} \label{P:canonical}
 Use notations as above with $G$ simple.  Suppose that $\rnk(G) \geq 2$, and exclude $G = A_2$ and $(G,P) = (B_2,P_1), (B_2,P_{1,2})$. Then $w(-\lambda) \in \Delta^-$ for any $w \in W^\fp(2)$.  For $\mu = -w\bullet \lambda$ with $\sfZ(\mu) > 0$, the canonical curved model of type $(\fg,\fp,\bbV_\mu)$ exists, and so $\fS_\mu = \fU_\mu$.  Moreover, any regular, normal parabolic geometry $(\cG \to M, \omega)$ of type $(G,P,\bbV_\mu)$ with $\dim(\inf(\cG,\omega)) = \fS_\mu$ is locally homogeneous about any $u \in \cG$ with $\kappa_H(u) \neq 0$.
 \end{prop}
 
 \begin{proof}
 This follows from \cite[Lemma 4.1.2]{KT2017} and local homogeneity follows from Lemma \ref{L:hom}.
 \end{proof}
 
 Under the hypotheses of Proposition \ref{P:canonical}, our problem of classifying submaximally symmetric models relative to $\bbV_\mu$ becomes that of classifying algebraic models $(\ff;\fg,\fp)$ with $0 \neq \kappa_H \in \bbV_\mu$ and $\dim(\ff) = \fU_\mu$.  From Proposition \ref{P:algCC}, $\fs = \tgr(\ff) \subseteq \fa^{\kappa_H}$.  The equality $\dim(\ff) = \fU_\mu = \dim(\fa^{\phi_0})$ forces $\fs = \fa^{\kappa_H}$ with $[\kappa_H] = g_0 \cdot [\phi_0]$ for some $g_0 \in G_0$.  Using the induced  $G_0$-action on $\ff$, we may henceforth assume that $[\kappa_H] = [\phi_0]$, and so $(\ff,[\cdot,\cdot]_\ff)$ satisfies
 \begin{align} \label{E:sa}
 \fs = \tgr(\ff) = \fa^{\phi_0}.
 \end{align}
 
 In summary, to establish Theorem \ref{T:main}, it suffices to answer the following question:
 \begin{framed}
 \begin{problem}  \label{P:problem} 
 When $\rnk(G) \geq 3$ or $(G,P) = (G_2,P_2)$, classify algebraic models $(\ff;\fg,\fp)$ with $0 \neq \kappa_H \in \bbV_\mu$ satisfying \eqref{E:sa}, up to the action by the residual subgroup $\Stab([\phi_0]) \ltimes P_+ \leq P$.)
 \end{problem}
 \end{framed}
 
 Over $\bbC$, we will show that the canonical curved model is the unique solution to Problem \ref{P:problem}.

 \subsection{The $(G_2,P_2)$ case} 
 \label{S:G2}
 
 We first answer Problem \ref{P:problem} in the $(G,P) = (G_2,P_2)$ case.  For $G_2$, recall:
 \begin{align}
 \Gdd{ww}{}, \quad (c_{ij}) = \begin{psmallmatrix} 2 & -1\\ -3 & 2 \end{psmallmatrix}, \quad \begin{cases}
 \alpha_1 = 2\lambda_1 - \lambda_2,\\
 \alpha_2 = -3\lambda_1 + 2\lambda_2
 \end{cases}, \quad \begin{cases}
 \lambda_1 = 2\alpha_1 + \alpha_2,\\
 \lambda_2 = 3\alpha_1 + 2\alpha_2
 \end{cases}.
 \end{align}
 Let $\{ \sfZ_1, \sfZ_2 \}$ be the dual basis to the simple roots $\{ \alpha_1, \alpha_2 \}$.  The highest weight is $\lambda = \lambda_2 = 3\alpha_1 + 2\alpha_2$.  The root diagram is given in Figure \ref{F:G2P}.
 Let $e_\alpha$ be a root vector for the root $\alpha \in \Delta$.
 
 \begin{center}
 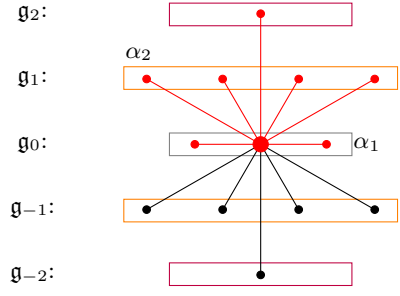
\begin{figure}[h]
 $\begin{array}{c}
 \begin{tikzpicture}[scale=1,baseline=-5pt]
  \draw[gray,yshift=49pt,dashed] (-1.2,0.15) rectangle (1.2,-0.15); 
  \draw[gray,yshift=25pt,dashed] (-1.8,0.15) rectangle (1.8,-0.15); 
  \draw[gray,dashed] (-1.2,0.15) rectangle (1.2,-0.15); 
  \draw[gray,yshift=-25pt,dashed] (-1.8,0.15) rectangle (1.8,-0.15); 
  \draw[gray,yshift=-49pt,dashed] (-1.2,0.15) rectangle (1.2,-0.15); 
    \draw[red] ( 0  , 1.732) -- (0,0);
    \filldraw[red] ( 0  ,  1.732) circle (0.05);
    \draw[red] (0, 0) -- (1.5,  0.866); 
    \filldraw[red] ( 1.5,  0.866) circle (0.05);
    \draw[red] (0, 0) -- (0.5,  0.866); 
    \filldraw[red] ( 0.5,  0.866) circle (0.05);
    \draw[red] (0, 0) -- (-0.5,  0.866); 
    \filldraw[red] ( -0.5,  0.866) circle (0.05);
    \draw[red] (0, 0) -- (-1.5,  0.866); 
    \node at (-1.6,1.2) {{\footnotesize $\alpha_2$}};
    \filldraw[red] ( -1.5,  0.866) circle (0.05);
    \draw[red] (0, 0) -- (0.866,  0); 
    \node at (1.4,0) {{\footnotesize $\alpha_1$}};
    \filldraw[red] (0.866,  0) circle (0.05);
    \draw (0, 0) -- (1.5,  -0.866); 
    \filldraw ( 1.5,  -0.866) circle (0.05);
    \draw ( 0  , -1.732) -- (0,0);
    \filldraw ( 0  ,  -1.732) circle (0.05);
    \draw (0, 0) -- (-1.5,  -0.866); 
    \filldraw ( -1.5,  -0.866) circle (0.05);
    \draw (0, 0) -- (-0.5,  -0.866); 
    \filldraw ( -0.5,  -0.866) circle (0.05);
    \draw (0, 0) -- (0.5,  -0.866); 
    \filldraw ( 0.5,  -0.866) circle (0.05);
    \draw[red] (0, 0) -- (-0.866,  0); 
    \filldraw[red] (-0.866,  0) circle (0.05);   
    \filldraw[color=red] (0,0) circle (0.1); 
    \node at (-3,1.732) {{\footnotesize $\fg_2$:}};
    \node at (-3,0.866) {{\footnotesize $\fg_1$:}};
    \node at (-3,0) {{\footnotesize $\fg_0$:}};
    \node at (-3,-0.866) {{\footnotesize $\fg_{-1}$:}};
    \node at (-3,-1.732) {{\footnotesize $\fg_{-2}$:}};
 \end{tikzpicture}
 \end{array}$
 \caption{$G_2$ root diagram with grading associated to $P_2$}
 \label{F:G2P}
 \end{figure}
 \end{center}
 
 We have $\sfZ = \sfZ_2$, which induces the grading $\fg = \fg_{-2} \op ... \op \fg_2$.  Moreover,
 \begin{align}
 \begin{array}{l} 
 \fg_0 = \langle \sfZ, h_{\alpha_1}, e_{\alpha_1}, e_{-\alpha_1} \rangle \cong \fgl_2,\\
 \fg_{-1} = \fg_{-\alpha_2} \op \fg_{-\alpha_1 - \alpha_2} \\
 \qquad\quad \op\, \fg_{-2\alpha_1 - \alpha_2} \op \fg_{-3\alpha_1 - \alpha_2},\\
 \fg_{-2} = \fg_{-3\alpha_1 - 2\alpha_2}, \\
 \end{array} \quad 
 \begin{array}{l}
 w = (21) \in W^\fp(2), \\
 \mu = -w \bullet \lambda = -7\lambda_1 + 4\lambda_2 = -2\alpha_1 + \alpha_2,\\
 H^2(\fg_-,\fg) = \Gdd{wx}{7,-4}, \\
 \phi_0 = e_{\alpha_2} \wedge e_{\alpha_1+\alpha_2} \otimes e_{-3\alpha_1-\alpha_2},\\
 \fa = \fg_- \op \fa_0, \quad \fa_0 = \langle \sfZ_1 + 2\sfZ_2 \rangle \op \fg_{-\alpha_1}.
 \end{array}
 \end{align}
 Let us now classify algebraic models $(\ff;\fg,\fp)$ with $\tgr(\ff) = \fa$. Let $T \in \ff^0$ with $\tgr_0(T) = \sfZ_1 + 2\sfZ_2$.  Since $(\sfZ_1 + 2\sfZ_2)(\alpha) \neq 0$ for all $\alpha \in \Delta(\fp_+)$, we use the $P_+$-action to normalize to $T = \sfZ_1 + 2\sfZ_2 \in \ff^0$.
 
 Defining $\fa^\perp = \langle \sfZ_2 \rangle \op \fg_{\alpha_1} \op \fg_+$, let $\fd \in \fa^* \otimes \fa^\perp$ be the associated deformation map (of positive degree).  The decomposition $\fg = \fa \op \fa^\perp$ is $\ad_T$-invariant, so $T \cdot \fd = 0$ by Lemma \ref{L:def-T-inv}, i.e.\ $\fd$ is a sum of weight vectors for weights that are multiples of $\mu = -2\alpha_1 + \alpha_2$.  The weights of $\fa^*$ and $\fa^\perp$ agree, and they are both:
 \begin{align} \label{E:a-wts}
 0, \quad \alpha_1, \quad \alpha_2, \quad \alpha_1 + \alpha_2, \quad 2\alpha_1 + \alpha_2, \quad 3\alpha_1 + \alpha_2, \quad 3\alpha_1 + 2\alpha_2.
 \end{align}
 Since $\mu$ has coefficients with respect to $\{ \alpha_1, \alpha_2 \}$ of {\em opposite sign}, there is no sum of two weights in \eqref{E:a-wts} that is: (i) a multiple of $\mu$, and (ii) has positive degree.  Thus, $\fd = 0$, and $\ff = \fa$ as filtered subspaces of $\fg$.

Now consider curvature $\kappa \in \ker(\partial^*)^1 \subset \bigwedge^2 \fp_+ \otimes \fg$.  Since $T \cdot \kappa = 0$ by Proposition \ref{P:algCC}, then we are interested in weights $\sigma$ (of 2-cochains) that are multiples of $\mu$:
 \begin{align} \label{E:sigmu}
 \sigma = r \mu = \alpha + \beta + \gamma, \qbox{where } \alpha, \beta \in \Delta(\fp_+) \mbox{ are distinct},\quad \gamma \in \Delta \cup \{ 0 \}, \quad r\geq 1.
 \end{align}
 (We have $r \geq 1$ since regularity and the final statement in \S \ref{S:parabolic} imply $\sfZ(\sigma) \geq \sfZ(\mu) \geq 1$.) Recall that $\lambda = \lambda_2 = 3\alpha_1 + 2\alpha_2$ and note that $-\lambda \leq \gamma < \sigma = r\mu$.  Then $-3 \leq \sfZ_1(-\lambda) \leq \sfZ_1(\sigma) = r \sfZ_1(\mu) = -2r$, and so $r \leq \frac{3}{2}$.  However, $\sfZ(\mu) = \sfZ_2(\mu) = 1$ and $\sigma$ has integer coefficients in the simple root basis, so the only possibility is $r = 1$.  Since $H^2(\fg_-,\fg) \cong \bbV_\mu$ is a $\fg_0$-irrep, then $\kappa$ must be a nonzero multiple of $\phi_0$.  Use $\Ad_{\exp(t\sfZ)}$ to rescale over $\bbC$ so that $\kappa = \phi_0$, so we obtain the canonical curved model.

Over $\bbR$, we may rescale to $\kappa = \pm \phi_0$.  Let us study the action by $G_0 \cong \GL(2,\bbR)$ more concretely.  Let $\{ x, y \}$ be the standard basis of $\bbR^2$, and $\{ \omega, \eta \}$ the dual basis.  Then $\fg_{-1} \cong S^3 \bbR^2$ as $G_0$-modules and we can identify 
 \[
 (e_{-\alpha_2},e_{-\alpha_1-\alpha_2},e_{-2\alpha_1-\alpha_2},e_{-3\alpha_1-\alpha_2}) = (x^3,x^2y,xy^2,y^3).
 \]
 We regard $\phi_0$ as a multiple of $\omega^3 \wedge \omega^2 \eta \otimes y^3$.  Hence, $A = \diag(a,b) \in G_0$ acts as $\phi_0  \mapsto a^{-5} b^2 \phi_0$, so taking $(a,b) = (-1,1)$ induces $\phi_0 \mapsto -\phi_0$.  Again, we obtain the canonical curved model.

 The underlying structures for regular, normal parabolic geometries of type $(G_2,P_2)$ are called {\sl $G_2$-contact geometries}.  See \cite{The2018} for $\kappa_H$ and a coordinate realization of a submaximally symmetric structure given in \cite[Table 8]{The2018}.  By uniqueness proved above, this corresponds to the canonical curved model.  We have shown:
 
 \begin{prop} \label{P:G2ct}
 There is a locally unique (complex or real) $G_2$-contact geometry that is submaximally symmetric ($\fS = 7$) about any point where harmonic curvature is nonvanishing.
 \end{prop}
 
 \subsection{Preparation for the general case and the twistor simplification}
 \label{S:cor-twist}
 
 The $\rnk(G) \geq 3$ case for Problem \ref{P:problem} is treated in a similar spirit to the $(G_2,P_2)$ case, but will require some further preparations.  We will need more details about $\mu = -w\bullet \lambda$ and $\tgr(\ff) = \fa^{\phi_0}$.
 First, observe that for $w = (jk)$, we have $-w\bullet 0 = \alpha_j + \sigma_j(\alpha_k)$ by \eqref{E:mu} and \eqref{E:phi0}.  If $\lambda = \sum_{i=1}^\ell r_i \lambda_i$, then \eqref{E:mu} becomes:
 \begin{align} \label{E:mu1}
 \mu &= -\lambda + (r_j+1)\alpha_j + (r_k+1)(\alpha_k - c_{kj} \alpha_j).
 \end{align}
 Second, from \cite[Thm. 3.3.3]{KT2017}, $\tgr(\ff) = \fa^{\phi_0}$ is the Tanaka prolongation of:
 \begin{align} \label{E:a0}
 \fann(\phi_0) = \ker(\mu) \op \bigoplus_{\gamma \in \Delta(\fg_{0,\leq 0})} \fg_\gamma\quad \subset\quad \fg_0.
 \end{align}
 This is the direct sum of $\ker(\mu) := \{ h \in \fh : \mu(h) = 0 \}$ and root spaces for the roots 
 \begin{align}
 \Delta(\fg_{0,\leq 0}) := \{ \alpha \in \Delta(\fg_0) : \sfZ_{J_\mu}(\alpha) \leq 0 \},
 \end{align}
 where $\sfZ_{J_\mu} := \sum_{i \in J_\mu} \sfZ_i$ is a  {\sl secondary grading} with respect to the set
 \begin{align} \label{E:Jmu}
 J_\mu := \{ i \in \{ 1,..., \ell \} \backslash I_\fp : \langle \mu, \alpha_i^\vee \rangle \neq 0 \}.
 \end{align}
 In \cite{KT2017}, the weight $\mu$ was encoded on a Dynkin diagram by inscribing over corresponding nodes the coefficients of $-\mu$ with respect to $\{ \lambda_i \}$.  The set $J_\mu$ corresponds to uncrossed nodes with a nonzero coefficient.

 \begin{example}
 Consider $(G,P) = (E_8,P_8)$.  Here, $\sfZ = \sfZ_8$, $\lambda = \lambda_8$, $w = (87) \in W^\fp(2)$, and $\fg_0 \cong \bbC \op E_7$.  Applying Kostant's theorem, we find that 
 \[
 H^2_+(\fg_-,\fg) \cong \bbV_\mu =  \Edd{wwwwwwww}{0,0,0,0,0,1,1,-4}
 \] with $\mu = -\lambda_6 - \lambda_7 + 4\lambda_8$, and so $J_\mu = \{ 6,7 \}$.  (Also, $\sfZ(\mu) = 1$.)  See \cite[\S 3.3]{KT2017} for more examples.
 \end{example}

 According to \cite{Cap2005}, any parabolic geometry can be lifted to a {\sl correspondence space}, and conversely a parabolic geometry may be descended to a {\sl twistor space} if a suitable curvature condition is satisfied.  (The latter amounts to viewing the given geometry of type $(G,Q)$ as a geometry of type $(G,P)$, where $Q \subset P \subset G$.)  These are categorical constructions, so symmetries are naturally mapped to symmetries.  We will not recall here the general theory developed in \cite{Cap2005}, but only summarize various results from \cite[\S 3.5]{KT2017} in order to emphasize a ``twistor simplification'' \eqref{E:twistor} relevant for our purposes.  The main reason for doing so is to assert \eqref{E:f1}, which facilitates the classification of filtered sub-deformations of $\fa^{\phi_0}$ in \S \ref{S:gen-pf}.
 
 Under the assumption $\kappa_H \in \bbV_\mu$, one may always descend to a {\sl minimal} twistor space.  Concretely, if $\mu = -w\bullet \lambda$, where $w = (jk) \in W^\fq(2)$, then \cite[Prop.3.5.1 \& Cor.3.5.2]{KT2017} indicates that we may instead view a given $(G,Q)$ geometry as a $(G,P)$ geometry, where
 \begin{align} \label{E:twistor}
 P = \begin{cases}
 P_j, & \mbox{if } c_{jk} < 0;\\
 P_{j,k}, & \mbox{if } c_{jk} = 0.
 \end{cases}
 \end{align}
 In \cite[Thm.3.5.4]{KT2017}, we showed that the Lie algebra structure of $\fa^{\phi_0}$ is unchanged with respect to the above change of parabolics, in spite of the grading change, cf.\ \cite[Example 3.5.5]{KT2017}.
 
 Normality of the geometry is preserved in passing to a correspondence or twistor space, but a priori regularity is not.
 
 \begin{example}
 For regular, normal geometries of type $(G_2,P_2)$, we have $\lambda = \lambda_2$ and $\sfZ = \sfZ_2$. Then $w = (21) \in W^\fp(2)$ yields $\mu = -7\lambda_1 + 4\lambda_2 = -2\alpha_1 + \alpha_2$, so $\sfZ(\mu) = 1$.  Viewed on the correspondence space as a $(G_2,P_{1,2})$ geometry, which has grading element $\widetilde\sfZ := \sfZ_1 + \sfZ_2$, the corresponding harmonic curvature would take values in a module with degree $\widetilde\sfZ(\mu) = -1$, i.e. regularity is not preserved.
 \end{example}
 
 Despite regularity not being preserved when passing upwards to a correspondence space, let us consider the passage downwards to the minimal twistor space.  In the {\em simple} setting, preservation of regularity upon such descent can be observed a posteriori through the tables compiled in \cite[Appendix C]{KT2017}.  We now give a uniform proof of this:

 \begin{lemma} 
 Let $\fg$ be a complex simple Lie algebra of $\rnk(\fg) \geq 2$ with highest weight $\lambda$, $\fq \subset \fg$ a parabolic subalgebra.  Fix $w = (jk) \in W^\fq(2)$ and $\mu = -w\bullet \lambda$.  Define $\fq \subset \fp \subset \fg$, where the parabolic subalgebra $\fp$ is defined by $I_\fq \supset I_\fp = \begin{cases} 
 \{ j \}, & c_{kj} < 0;\\
 \{ j,k \}, & c_{kj} = 0.
 \end{cases}$ Let $\sfZ,\bar\sfZ$ be the grading elements corresponding to $\fq,\fp$ respectively.  If we have $\sfZ(\mu) > 0$, then $\bar\sfZ(\mu) > 0$. 
 \end{lemma}

 \begin{proof} 
 From \eqref{E:mu1}, we have $\sfZ_s(\mu) = \sfZ_s(-\lambda) < 0$ for any $s \neq j,k$.  (Recall that all coefficients of  $\lambda$ in the simple root basis are {\em positive}.)  We have $j \in I_\fq$.   Suppose that (i) $c_{jk} = 0$ (hence, $k \in I_\fq$), or (ii) $c_{jk} < 0$ and $k \not\in I_\fq$.  In either case, $0 < \sfZ(\mu) = \bar\sfZ(\mu) + \sfZ_{I_\fq \backslash \{ j,k \}} (\mu)$, where $\sfZ_{I_\fq \backslash \{ j,k \}} := \sum_{s \in I_\fq \backslash \{ j,k \}} \sfZ_s$, so $\bar\sfZ(\mu) > 0$ since $\sfZ_{I_\fq \backslash \{ j,k \}}(\mu) < 0$.
  
 From \S\ref{S:can-submax}, it remains to consider the case $c_{jk} < 0$ and $k \in I_\fq$.  Then
  \begin{align} \label{E:Zk}
  0 < \sfZ(\mu) = \bar\sfZ(\mu) + \sfZ_k(\mu) + \sfZ_{I_\fq \backslash \{ j,k \}} (\mu)
  \end{align}
 From \eqref{E:mu1}, note that 
 \begin{align} \label{E:mu2}
 \sfZ_k(\mu) = \sfZ_k(-\lambda) + r_k+1.
 \end{align}
 If $r_k = 0$, then $\sfZ_k(\mu) \leq -1 + 0 + 1 = 0$.  As above, $\bar\sfZ(\mu) > 0$ and we are done.  So let us suppose that $r_k > 0$.  We can examine all such possibilities from knowledge of  the well-known highest roots of simple Lie algebras:
 \begin{align} \label{E:hw}
 \begin{array}{|c|c|c|c|c|c|c|c|c|c} \hline
  A_\ell\, (\ell \geq 1) & B_\ell \, (\ell \geq 3) & C_\ell\, (\ell \geq 2) & D_\ell \, (\ell \geq 4) & E_6 & E_7 & E_8 & F_4 & G_2\\  \hline
  \lambda_1 + \lambda_\ell & \lambda_2 & 2\lambda_1 & \lambda_2 & \lambda_2 & \lambda_1 & \lambda_8 & \lambda_1 & \lambda_2\\ \hline
 \end{array}
 \end{align}
 If $\fg$ is not type A or C, then from \eqref{E:hw}, we have $r_k = 1$, and it is well-known that $\sfZ_k$ yields a contact grading on $\fg$.  So $\sfZ_k(\lambda) = 2$, $\sfZ_k(\mu)  = -2 + 1 + 1 = 0$ from \eqref{E:mu2}, and $\bar\sfZ(\mu) > 0$ follows from \eqref{E:Zk}.  For the type A and C cases, we show that $\bar\sfZ(\mu) > 0$ independent of the hypothesis on $\sfZ(\mu)$:
 \begin{enumerate}
 \item Type C: We have $r_k = 2$, $k=1$, and $j=2$.  Since $\lambda = 2\lambda_1 = 2\alpha_1 + ... + 2\lambda_{\ell-1} + \lambda_\ell$, then from \eqref{E:mu1}, we have $\bar\sfZ(\mu) = \sfZ_j(\mu) = \sfZ_j(-\lambda) + r_j + 1 - (r_k+1) c_{kj} \geq -2 + 0 + 1 - 3 c_{kj} \geq 2$.
 \item Type A: We have $r_k=1$ and using a Dynkin diagram symmetry, we may assume $k=1$, so $j=2$.  Since $\lambda = \alpha_1 + ... + \alpha_\ell$, we have $\bar\sfZ(\mu) = \sfZ_j(\mu) = \sfZ_j(-\lambda) + r_j + 1 - (r_k+1) c_{kj} \geq -1 + 0 + 1 - 2 c_{kj} = 2$.
 \end{enumerate}
 \end{proof}
 
 Now, because of \cite[Prop.3.4.7]{KT2017} (see also \cite[Defn 3.4.1]{KT2017}), the ``twistor simplification'' implies that after moving to the larger parabolic subgroup indicated in \eqref{E:twistor} and the corresponding grading, we get $\fa^{\phi_0}_+ = 0$.  Combining this with \eqref{E:sa}, we obtain:
 \begin{align} \label{E:f1}
 \ff^1 = 0.
 \end{align}

 \subsection{Proof of the main theorem} 
 \label{S:gen-pf}
 
 Let us turn now to the proof of Theorem \ref{T:main}.

 \begin{lemma} \label{L:mu}
 Let $\fg$ be a complex simple Lie algebra with $\ell := \rnk(\fg) \geq 3$ and $\lambda$ its highest root.  Let $w = (jk) \in W^\fp(2)$ such that $\mu = -w\bullet \lambda$ satisfies $\sfZ(\mu) > 0$.  Then:
 \begin{enumerate}[label=(MU{{\arabic*}}$)$]
 \item[\mylabel{mu1}{(MU1)}] $\mu = \sum_{i=1}^\ell m_i \alpha_i$ has coefficients $m_i$ of opposite sign.  More precisely, $m_i < 0$, $\forall i \neq j,k$, and either $m_j > 0$ or $m_k > 0$.
 \item[\mylabel{mu2}{(MU2)}] $\exists H_0 \in \ker(\mu)$ with $f(H_0) \neq 0$ for all $f = \alpha + \beta$ with $(\alpha,\beta) \in \cR := \Delta^+ \times (\Delta^+ \cup \{ 0 \})$.
 \end{enumerate}
 \end{lemma}
 
 \begin{proof}
 From \eqref{E:mu1}, $\mu \equiv -\lambda\, \mod \{ \alpha_j, \alpha_k\}$.  Since $\fg$ is simple, then {\em all} coefficients of $\lambda$ with respect to the basis of simple roots $\{ \alpha_i \}_{i=1}^\ell$ are {\em strictly positive}.  Since $\ell \geq 3$, then $m_i = \sfZ_i(\mu) = \sfZ_i(-\lambda) < 0$ for all $i \neq j,k$.  At least one of $m_j = \sfZ_j(\mu)$ or $m_k = \sfZ_k(\mu)$ must be positive, since $\sfZ(\mu) > 0$ by hypothesis.

 Fix any $(\alpha,\beta) \in \cR$, and $f = \alpha+\beta > 0$, so by the first claim, $\mu$ is not a multiple of $f$.  Thus, $\ker(\mu)$ and $\ker(f)$ are distinct hyperplanes in $\fh$.  Their sum must be $\fh$, while $\Pi_f := \ker(\mu) \cap \ker(f)$ is a hyperplane in $\ker(\mu)$.  Since $\Delta^+$ is finite, the finite union $\bigcup_{(\alpha,\beta) \in \cR} \Pi_{\alpha + \beta}$ has non-empty complement in $\ker(\mu)$ (being the finite intersection of open sets $\fh\, \backslash\, \Pi_{\alpha+\beta}$).  Picking $H_0$ in this (open) complement completes the proof.
 \end{proof}
 
 Assume the hypotheses of Lemma \ref{L:mu}.  From the previous subsections, we have reduced our submaximal symmetry classification problem to studying algebraic models $(\ff;\fg,\fp)$ with 
 \begin{align}
 \im(\kappa_H) \subseteq \bbV_\mu, \quad
 \fs = \tgr(\ff) = \fa^{\phi_0} =: \fa,
 \end{align}
 where $\phi_0 \in \bbV_\mu$ is given by \eqref{E:phi0}.  We will classify these up to the action of $\Stab([\phi_0]) \ltimes P_+ \leq P$.  Moreover, we may assume the twistor simplification, which implies that $\ff^1 = 0$, where we have moved to the grading associated with the larger parabolic subgroup indicated in \eqref{E:twistor}. \\

 \underline{Step 1: Using the $P_+$-action, normalize $\ff$ so that $H_0 \in \ff^0$.}\\
 
 As in Lemma \ref{L:mu}, fix $H_0 \in \ker(\mu) \subset \fann(\phi_0) \subset \fa$.  Let $H \in \ff^0 \subset \fg$ with leading part $\tgr_0(H) = H_0$, so $H = H_0 + H_+$, where $H_+ := \fd(H_0) \in \fp_+$.  If $H_+ \neq 0$, let $0 \neq H_r := \tgr_r(H_+) \in \fg_r$ for some minimal $r \geq 1$.  Let us normalize $H$ via the $P_+$-action.  Letting $X \in \fg_r$, we have:
 \begin{align}
 \Ad_{\exp(X)}(H) = \exp(\ad_X)(H) = H + [X,H] + \ldots = H_0 + H_r - [H_0,X] + \ldots,
 \end{align}
 where the dots indicate terms of degree $> r$.  Fixing root vectors $e_\alpha \in \fg_\alpha$, we have $H_r = \sum_{\alpha \in \Delta(\fg_r)} c_\alpha e_\alpha$.  By \ref{mu2} in Lemma \ref{L:mu}, $\alpha(H_0) \neq 0$ $\forall \alpha \in \Delta^+$, so defining $X := \sum_{\alpha \in \Delta(\fg_r)} \frac{c_\alpha}{\alpha(H_0)} e_\alpha$, we have $H_r - [H_0,X] = 0$.  Redefining $\Ad_{\exp(X)}(\ff)$ as $\ff$ and  $\Ad_{\exp(X)}(H)$ as $H$, the latter has $H_+$ with leading part of degree $r+1$.  Inductively, we may normalize $H_+= \fd(H_0) = 0$, and so $H = H_0 \in \fh \cap \ff^0$.  Since $\alpha(H_0) \neq 0$ for all $\alpha \in \Delta(\fp_+)$, the $P_+$-part of the structure group has been completely reduced.\\
 
 \underline{Step 2: Observe that $\ker(\mu) \subset \ff^0$}\\
 
 Fix any $0 \neq H_0' \in \ker(\mu) \subset \fh$, so $H_0' \in \fa$.  Write $H' = H_0' + H_+' \in \ff$ with $H_+' = \fd(H_0') \in \fp_+ = \fg^1$.  By \ref{D:M2} from Definition \ref{D:alg-model}, we have:
 \begin{align}
 [H_0,H']_\ff = [H_0,H'] = [H_0,H'_0 + H_+'] = [H_0,H_+'] \in \fp_+ \cap \ff = \ff^1 = 0,
 \end{align}
 where the twistor simplification was invoked for the last equality.  Since \ref{mu2} implies $\alpha(H_0) \neq 0$ for all $\alpha \in \Delta(\fp_+)$, then necessarily $H_+' = 0$.  Thus, $\ker(\mu) \subset \ff^0$.\\
 
 \underline{Step 3: Show that $\ff = \fa$ as subspaces of $\fg$.}\\

 Recall $J_\mu$ from \eqref{E:Jmu} and the secondary grading $\sfZ_{J_\mu}$.  From \eqref{E:a0}, we have $\fg = \fa \op \fa^\perp$ where
 \begin{align}
 \begin{split}
 &\fa = \fg_- \op \fa_0, \qbox{where} \fa_0 = \ker(\mu) \op \bigoplus_{\gamma \in \Delta(\fg_{0,\leq 0})} \fg_\gamma, \\
 &\fa^\perp := \ker(\mu)^\perp \op \fg_{0,+} \op \fg_+, \qbox{where} \fg_{0,+} := \bigoplus_{\gamma \in \Delta(\fg_{0,+})} \fg_\gamma,
 \end{split}
 \end{align}
  and $\ker(\mu)^\perp$ is a 1-dimensional complement to $\ker(\mu)$ inside $\fh$.  Both $\fa$ and $\fa^\perp$ are $\fh$-invariant, so in particular they are invariant under $\ker(\mu)$.  Defining the associated deformation map $\fd : \fa \to \fa^\perp$, Lemma \ref{L:def-T-inv} implies that $H \cdot \fd = 0$, $\forall H \in \ker(\mu) \subset \ff^0$, so $\fd$ lies in the direct sum of weight spaces of $\fa^* \otimes \fa^\perp$ for weights that are {\bf multiples of $\mu$}.
 
 Note $\Delta^- \subset \Delta(\fa)$, so let $\alpha \in \Delta^+$ and examine $\fd(e_{-\alpha})$.  From \eqref{E:a0}, we have $\fg_{-\alpha} \subset \fa$, and the weights of $e_{-\alpha}^* \otimes \fa^\perp$ are of the form $\alpha + \gamma$, where $\gamma \in \Delta^+(\fa^\perp) \cup \{ 0 \}$.  These all have {\em non-negative} coefficients in the simple root basis. By \ref{mu1}, these weights cannot be multiples of $\mu$.  Hence, $\fd(e_{-\alpha}) = 0$, i.e.\ $e_{-\alpha} \in \ff$.  (This argument is very similar to the $(G_2,P_2)$ case from \S \ref{S:G2}.)
 
 For our Step 3 claim, it suffices to consider $\alpha \in \Delta^+(\fa) = \Delta^+(\fg_{0,0})$ and show that $\fd(e_\alpha) = 0$, $\forall \alpha \in \Delta^+(\fg_{0,0})$. First recall that $w = (jk) \in W^\fp(2)$ as in Lemma \ref{L:mu}.  We claim that we may assume
 \begin{align} \label{E:Zmu}
 \sfZ_j(\mu) > 0, \quad \sfZ_s(\mu) = \sfZ_s(-\lambda) < 0, \quad s \neq j,k.
 \end{align}
 Via the twistor simplification, we have either: (a) $I_\fp = \{ j \}$, hence $\sfZ_j(\mu) = \sfZ(\mu) > 0$; or (b) $I_\fp = \{ j,k \}$ with $c_{jk} = 0$, hence $0 < \sfZ(\mu) = \sfZ_j(\mu) + \sfZ_k(\mu)$, so either $\sfZ_j(\mu) > 0$ or $\sfZ_k(\mu) > 0$.  Since $c_{jk} = 0$, then swap $j,k$ if necessary to assume that $\sfZ_j(\mu) > 0$.  Since $\ell \geq 3$, \eqref{E:mu1} implies the rest of \eqref{E:Zmu}.

  Since $\fd$ has positive degree, then $\fd(e_\alpha) \in \fp_+$, so let us consider a weight $\gamma - \alpha$ for $\gamma \in \Delta(\fp_+)$ corresponding to a possible term $e_\alpha^* \otimes e_\gamma$ in $\fd$.  Using $J_\mu$, we have two cases:
  \begin{enumerate}
 \item $J_\mu \backslash \{ k \} \neq \emptyset$: Since $\sfZ_j(\alpha) = 0$, then $\sfZ_j(\gamma - \alpha) = \sfZ_j(\gamma) > 0$, while for any $i \in J_\mu \backslash \{ k \}$, we have $\sfZ_i(\gamma - \alpha) = \sfZ_i(\gamma) > 0$.  By \eqref{E:Zmu}, $\gamma - \alpha$ cannot be a multiple of $\mu$.
 \item $J_\mu\backslash \{ k \} = \emptyset$: Since $\ell \geq 3$, fix any $s \neq j,k$ and note that $c_{js},c_{ks},c_{kj} \leq 0$ (by standard properties of Cartan matrices).  By definition of $J_\mu$, we have 
 $\langle \mu, \alpha_s^\vee \rangle = 0$.  Recalling that $r_i \geq 0$ for all $i$, \eqref{E:mu1} implies:
 \begin{align}
 0 = \langle -\mu, \alpha_s^\vee \rangle = r_s - (r_j+1)c_{js} - (r_k+1)(c_{ks} - c_{kj} c_{js}) \geq r_s \geq 0.
 \end{align}
 Hence, $r_s = 0$, $c_{js} = c_{ks}=0$, i.e.\ every $s \neq j,k$ is {\em not} connected in the Dynkin diagram to either $j$ or $k$.  Since $\fg$ is simple (with $\rnk(\fg) \geq 3$), its Dynkin diagram is connected, so this is a contradiction, i.e.\ this case is vacuous.
 \end{enumerate}
 We conclude that $\fd(e_\alpha) = 0$, $\forall \alpha \in \Delta^+(\fg_{0,0})$, and hence $\fd = 0$.  Thus, $\ff = \fa$ as subspaces of $\fg$.\\
 
 \underline{Step 4: Study curvature $\kappa$}\\
 
 By \ref{D:M3} and Proposition \ref{P:algCC}, we have $\kappa \in \ker(\partial^*)^1 \subset \bigwedge^2 \fp_+ \otimes \fg$ and $\ff^0 \cdot \kappa = 0$.  Since $\ker(\mu) \subset \ff^0$, then $\kappa$ is valued in the direct sum of weight spaces of $\ker(\partial^*)^1$ for weights $\sigma = r\mu = \alpha + \beta + \gamma$ with $\alpha, \beta \in \Delta(\fp_+)$ and $\gamma \in \Delta \cup \{ 0 \}$.  For the same reasons there (regularity and the final statement in \S \ref{S:parabolic}), we again have $r \geq 1$.  Let us show that $r \leq 1$.  Write the highest weight of $\fg$ as $\lambda = \sum_i n_i \alpha_i$, where $n_i > 0$ for all $i$ since $\fg$ is simple.  Since $-\lambda$ is the lowest root of $\fg$, then $-\lambda \leq \gamma < \sigma$.  Thus, for any $i \neq j,k$, 
 \begin{align}
 -n_i = \sfZ_i(-\lambda) \leq \sfZ_i(\gamma) \leq \sfZ_i(\sigma) = r\sfZ_i(\mu) = -r n_i,
 \end{align}
 where the last equality follows from \eqref{E:mu1}. Since $n_i > 0$, then $r \leq 1$ follows.  Thus, $r=1$ and so $\kappa$ has weight $\sigma = \mu$.  The multiplicity of $\mu$ (lowest weight) is the same as that occurring in the $\fg_0$-irrep $\bbV_\mu$, i.e.\ multiplicity one, by Kostant's theorem.  Under the identification with harmonic 2-cochains, $\kappa$ must be a nonzero multiple of $\phi_0$.  Using $\Ad_{\exp(t\sfZ)}$, we may do a complex rescaling to arrange $\kappa = \phi_0$.  Thus, we have obtained the canonical curved model.
 
 Working with split-real geometries, we similarly arrive at $\kappa$ being a nonzero multiple of $\phi_0$ using almost exactly the same arguments as in the complex case.  The only part that differs concerns the use of \cite[Prop.3.1.1]{KT2017} to assert \eqref{E:TPmax} and the subsequent statement characterizing equality there.  A key ingredient for that Proposition is that {\bf $\cO = G_0 \cdot [\phi_0]$ is the unique closed $G_0$-orbit in $\bbP(\bbV_\mu)$}, and this orbit is of minimal dimension.  This is a well-known result in the complex setting, and the result remains true in the split-real setting -- see \cite[Cor.1]{Winther2023}.  All other arguments in \cite[Prop.3.1.1]{KT2017} and this section are exactly the same to arrive to $\kappa$ being a nonzero multiple of $\phi_0$.
  
 Finally, a real scaling using $\Ad_{\exp(t\sfZ)}$ normalizes $\kappa = \pm \phi_0$.  The algebraic models are $P$-equivalent if and only if there exists $g_0 \in G_0$ such that $g_0 \cdot \phi_0 = -\phi_0$.  The proof of Theorem \ref{T:main} is complete.

 \section{Examples}
 \label{S:examples}
 
 In this final section, we apply Theorem \ref{T:main} and give concrete examples of submaximally symmetric parabolic geometries, expressed as underlying geometric structures.  Implicit here are known equivalences of categories, in particular the parabolic geometry types $(G,P)$ associated to given structures.  We do not provide details here, but instead refer the reader to \cite{CS2009}.
 
 We will use the following notation.  Let $E_{ij}$ denote the standard square matrix (of size to be specified) with a 1 in the $(i,j)$-position and $0$ elsewhere.  We continue to use $\lambda$ for the highest weight of $\fg$, and $\phi_0$ for a lowest weight vector of a $\fg_0$-irreducible submodule of $H^2_+(\fg_-,\fg)$, obtained via Kostant's theorem (\S \ref{S:can-submax}).

 \subsection{Projective structures}
 
 On a manifold $M^n$, two torsion-free affine connections are equivalent if and only if they admit the same unparametrized geodesics, and an equivalence class $[\nabla]$ is called a {\sl projective structure}.  These well-known structures underlie geometries of type $(G,P) = (A_n,P_1)$, for which $\fM = (n+1)^2 - 1$, and harmonic curvature corresponds to the {\sl projective Weyl curvature}.  Here, $G_0 \cong \GL(n,\bbF)$ (for $\bbF = \bbR$ or $\bbC$) realized as matrices of the form $A = \begin{psmallmatrix} \det(A_0)^{-1} & 0\\ 0 & A_0 \end{psmallmatrix} \in \tMat_{n+1}(\bbF)$, where $A_0 \in \GL(n,\bbF)$.  In \cite{KT2017}, we found that $\fS = (n-1)^2+4$ for $n \geq 3$, realized in particular by the Egorov projective structure \cite{Egorov1951}, \cite[(5.11)]{KT2017}.  We can now assert:
 
 \begin{cor} \label{T:proj} 
 Let $n \geq 3$, and $(M^n,[\nabla])$ a submaximally symmetric projective structure with non-vanishing projective Weyl curvature at $x \in M$.  Then about $x$, $(M^n,[\nabla])$ is locally equivalent to the Egorov projective structure (in either the real or complex settings).
 \end{cor}

 \begin{proof} 
 Using Theorem \ref{T:main}, we immediately conclude the result over $\bbF = \bbC$, so consider $\bbF = \bbR$.
 Using $w = (12) \in W^\fp(2)$ and $\lambda = \lambda_1 + \lambda_n$, we obtain $\phi_0 = e_{\alpha_1} \wedge e_{\alpha_1 + \alpha_2} \otimes e_{-\alpha_2 - ... - \alpha_n} = E_{12} \wedge E_{13} \otimes E_{n+1,2}$, where $E_{ij} \in \tMat_{n+1}(\bbR)$.    Letting $A = \diag(a_1,...,a_{n+1}) \in G_0$ where $a_1 = (a_2 \cdots a_{n+1})^{-1}$, we get $A \cdot \phi_0 = \frac{a_1^2 a_{n+1}}{a_2^2 a_3} \phi_0$.  Since $n \geq 3$, then setting $a_2 = ... = a_n = 1$ and $a_1 = a_{n+1} = -1$, we get $A \cdot \phi_0 = -\phi_0$.  Invoking Theorem \ref{T:main} now gives the result.
 \end{proof}
 
 \begin{remark} \label{R:R}
 Over $\bbR$, some attention should be given to the choice of Lie group $G$.  Choosing $G = A_n := \SL(n+1,\bbR)$ with $G_0$ as above, the induced $G_0$-action on $\fg_{-1}$ is $v \mapsto Bv$, where $B = R_0 \det(R_0)$, so $\det(B) = \det(R_0)^{n+1}$, which is always positive when $n$ is odd.  In these cases, one is in fact working with {\em oriented} manifolds.  In the unoriented setting, one could work with $G = \PGL(n+1,\bbR)$ (i.e.\ $\GL(n+1,\bbR)$ modulo its centre $\cZ(\GL(n+1,\bbR))$, used as in \cite[Prop.4.1.5]{CS2009}) or use $G = \widehat{SL}(n+1,\bbR) := \{ R \in \tMat_{n+1}(\bbR): \det(R) = \pm 1 \}$ when $n$ is odd.
 \end{remark}

 \subsection{2nd order ODE systems}
 \label{S:2ODEsys}

 Any system $\ddot{x}^i = F^i(t,x^j,\dot{x}^j)$, $1 \leq i \leq m$ of 2nd order ODE in $m \geq 2$ dependent variables (viewed up to point transformations) admits an equivalent description as a (regular, normal) parabolic geometry of type $(A_{m+1},P_{1,2}) = (\PGL(n+1,\bbF),P_{1,2})$.  (In \cite[\S 4.4.3]{CS2009}, these are formulated as {\sl generalized path geometries}.  When $m \geq 3$ (or $m=1$), these can all be locally realized as 2nd order ODE systems, while for $m=2$, we additionally have the constraint that $\kappa_H$ vanishes in degree $+1$.)  We have $\fM = (m+2)^2 - 1$, locally uniquely realized by the trivial ODE $\ddot{x}^1 = ... = \ddot{x}^m = 0$.  Here, 
 \begin{align}
 G_0 = \left\{ \begin{psmallmatrix} a_1 & 0 & 0\\ 0 & a_2 & 0\\ 0 & 0 & A_0 \end{psmallmatrix}: A_0 \in \GL(m,\bbF),\, a_i \in \bbF^\times \right\} \quad\mod \cZ(\GL(n+1,\bbF)),
 \end{align}
 and harmonic curvature decomposes into two components: {\sl Fels curvature} (degree +3) and {\sl Fels torsion} (degree +2).  Referring to \cite[\S 5.3 and \S 5.4]{KT2018}, we have (using $\lambda = \lambda_1 + \lambda_{m+1}$ and $E_{ij} \in \tMat_{m+2}(\bbF)$ and notation $\fS_\mu$, $\fU_\mu$ from \S \ref{S:can-submax}):
 \begin{itemize}
 \item $w = (21)$ (``Segr\'e branch'', i.e.\ vanishing Fels torsion): $\mu_1 := -w\bullet \lambda = 4\lambda_2 - 3\lambda_3 - \lambda_{m+1}$ has $\sfZ(\mu_1) = +3$, and $\fS_{\mu_1} = \fU_{\mu_1} = m^2 + 5$, realized in the Segr\'e branch by:
 \begin{align} \label{E:2ODE-Segre}
 \ddot{x}^1 = ... = \ddot{x}^{m-1} = 0, \quad \ddot{x}^m = (\dot{x}^1)^3.
 \end{align}
 \item $w = (12)$ (``projective branch'', i.e.\ vanishing Fels curvature): $\mu_2 := -w\bullet \lambda = 4\lambda_1 - \lambda_2 - \lambda_3 - \lambda_{m+1}$ has $\sfZ(\mu_2) = +2$ and $\fS_{\mu_2} = \fU_{\mu_2} = m^2 + 4$, realized in the projective branch by the geodesic equations of the Egorov projective structure:
 \begin{align} \label{E:2ODE-proj}
 \ddot{x}^i = 2 x^1 \dot{x}^1 \dot{x}^2 \dot{x}^i, \quad 1 \leq i \leq m.
 \end{align}
  Using the point transformation $(\tilde{t}, \tilde{x}^1, \tilde{x}^2, ..., \tilde{x}^m) = (x^1,t + \half (x^1)^2 x^2, x^2, ...,  x^m)$ given in \cite{BLP2021}, a simpler alternative model to \eqref{E:2ODE-proj} is
 \begin{align} \label{E:2ODE-proj2}
 \ddot{x}^1 = x^2, \quad \ddot{x}^2 = ... = \ddot{x}^m = 0.
 \end{align}
 (All ODE in the projective branch are geodesic equations for some projective structure, and Theorem \ref{T:proj} asserts the classification of submaximal symmetry models in this branch.)
 \end{itemize}

 Uniqueness of the submaximally symmetric ODE \eqref{E:2ODE-Segre} and \eqref{E:2ODE-proj2} was recently asserted in \cite[Theorems 2 \& 3]{BLP2021} without proof.  Applying our Theorem \ref{T:main}, we obtain:
 
 \begin{cor}
 Let $m \geq 2$.  Over $\bbF = \bbR$ or $\bbC$, suppose that a given 2nd order ODE system $\ddot{x}^i = F^i(t,x^j,\dot{x}^j)$, $1 \leq i \leq m$ is submaximally symmetric, i.e.\ it has point symmetry algebra of dimension $\fS = m^2+5$.  Then the system has vanishing Fels torsion everywhere, and about any point where Fels curvature is non-vanishing, the system is locally point equivalent to \eqref{E:2ODE-Segre}.
 
 Within the projective branch (i.e.\ vanishing Fels curvature), about any point where Fels torsion is non-vanishing, any submaximally symmetric system (realizing $\fS_{\mu_2} = m^2 + 4$) is locally point equivalent to \eqref{E:2ODE-proj2}.
 \end{cor}
 
 \begin{proof}
 Note that we have $\fU = \max\{ \fU_{\mu_1}, \fU_{\mu_2} \} = \fU_{\mu_1} = m^2 + 5$ and $\fS = \fS_{\mu_1} = m^2 + 5$.  Since $\fS = \fU$, then local homogeneity follows from Lemma \ref{L:hom}. Write $\phi = \phi_1 + \phi_2$ for $\phi_1 \in \bbV_{\mu_1}$ and $\phi_2 \in \bbV_{\mu_2}$, where $\bbV_{\mu_i}$ are the $\fg_0$-irreducible submodules of $H^2_+(\fg_-,\fg)$ corresponding to Fels curvature and Fels torsion respectively.  We have $\fa^\phi \subset \fa^{\phi_2}$, which has maximal dimension $\fU_{\mu_2} = m^2+4$ when $\phi_2 \neq 0$.  By \eqref{E:s-Kh} and the symmetry dimension being $m^2+5$, the Fels torsion must vanish everywhere.
 
  We now invoke Theorem \ref{T:main}.  Using $w = (21) \in W^\fp(2)$, we find that $\phi_0 = e_{\alpha_2} \wedge e_{\alpha_1 + \alpha_2} \otimes e_{-\alpha_3 - ... - \alpha_{m+1}} = E_{23} \wedge E_{13} \otimes E_{m+2,3}$.  For $A = \diag(a_1,...,a_{m+2}) \in G_0$, we get $A \cdot \phi_0 = \frac{a_1a_2a_{m+2}}{a_3^3} \phi_0$.  Setting $a_1 = ... = a_{m+1} = 1$ and $a_{m+2} = -1$, we get $A \cdot \phi_0 = -\phi_0$, so uniqueness now follows from Theorem \ref{T:main}. Our final statement reformulates Corollary \ref{T:proj} via the correspondence space construction (\S \ref{S:cor-twist}).
 \end{proof}

 We remark that in \cite{KT2017}, we used $G = \SL(m+2,\bbR)$ instead of $G = \PGL(m+2,\bbR)$.  This small change does not affect $\fS$ and $\fU$, but the notion of point equivalence is slightly restricted with the former, as we now explain.  Consider $A = \begin{psmallmatrix} a_1 & 0 & 0\\ 0 & a_2 & 0\\ 0 & 0 & A_0 \end{psmallmatrix} \in G_0$ with $a_1 a_2 \det(A_0) = 1$.  The ODE structure is modelled on $\fg_{-1}$, which is split into the direct sum of $\langle E_{21} \rangle $ (corresponding to the line field spanned by the total derivative $D_t := \partial_t + \dot{x}^i \partial_{x^i} + F^i \partial_{\dot{x}^i}$) and $\langle E_{32},..., E_{m+2,2} \rangle$ (corresponding to $\langle \partial_{\dot{x}^i} \rangle$).  On $\fg_{-1}$, $A$ induces:
 \begin{align}
 (\ell,v) \mapsto \left( c \ell,  B_0 v \right), \quad c = \tfrac{a_2}{a_1}, \quad B_0 = A_0 a_2^{-1}.
 \end{align}
 But then $\det(B_0) = \det(A_0) a_2^{-m} = \frac{1}{a_1 a_2^{m+1}} = \frac{c}{a_2^{m+2}}$.  When $m$ is even, the signs of $c$ and $\det(B_0)$ are aligned, and the point transformation $(t,x^i) \mapsto (-t,x^i)$ would not be an admissible equivalence.
 
 If we consider $G = \SL(m+2,\bbR)$, then for $m \geq 3$, setting $a_i = 1$ for $i \neq 3,4$ except $a_3 = a_4 = -1$ yields $A \cdot \phi_0 = -\phi_0$.  When $m=2$, we have $A \cdot \phi_0 = \frac{a_1 a_2 a_4}{a_3^3} \phi_0 = \frac{1}{a_3^4} \phi_0$, and no $A \in G_0$ exists with $A \cdot \phi_0 = -\phi_0$.  In this case, $\ddot{x} = 0$, $\ddot{y} = \pm \dot{x}^3$ would be inequivalent submaximally symmetric models.

 \subsection{Conformal structures}  
 
 Given a smooth manifold $M^n$ with $n \geq 3$ and a metric $g$ of signature $(p,q)$, we let $[g] := \{ \lambda g \,|\, \lambda : M \to \bbR^+ \mbox{ smooth} \}$, and refer to $(M^n,[g])$ as a {\sl conformal structure}.  This admits an equivalent description as a parabolic geometry of type $(G,P) = (\SO(p+1,q+1),P_1)$, where $P_1$ is the stabilizer of a null line in $\bbR^{p+1,q+1}$, so $\fM = \binom{n+2}{2}$.  Restrict now to $n \geq 4$.  See \cite{DT2014} for $\fS$ in the Riemannian / Lorentzian cases, which are exceptional.  In non-Riemannian / non-Lorentzian signatures, \cite[\S5.1]{KT2017} indicates $\fS = \binom{n-1}{2} + 6$, realized by $[\bg]$, with $\bg$ given by the direct product of a flat Euclidean metric of signature $(p-2,q-2)$ and the $(2,2)$-metric
 \begin{align}
 g_{{\rm pp}}^{(2,2)} = y^2 dw^2 + dw dx + dy dz.
 \end{align}
 Restrict now to the split-real form, so $|p-q| \leq 1$.  We view $\fg \subset \tMat_{n+2}(\bbR)$ as matrices that are skew with respect to the anti-diagonal, and let $G_0$ consist of block diagonal matrices with blocks $(\lambda, C, \lambda^{-1})$ with $\lambda > 0$ and $C \in \SO(p,q)$, and so the $G_0$-action on $\fg_{-1}$ is given by $x \mapsto \lambda^{-1} C x$.  In particular, any scalar product on $\fg_{-1}$ is only positively rescaled, so any conformal structure $[g]$ and its ``negative'' $[-g]$ are inequivalent. Together with Theorem \ref{T:main}, we deduce:

 \begin{cor}
 Let $n = p+q \geq 4$ and $|p-q| \leq 1$.  Suppose that a conformal structure of signature $(p,q)$ is submaximally symmetric, i.e.\ its conformal symmetry algebra has dimension $\fS = \binom{n-1}{2} + 6$.  Then about any point where the Weyl curvature is non-vanishing, the structure is locally conformally equivalent to one of the two models $[\bg]$ or $[-\bg]$ described above.
 \end{cor}

 The split-signature assumption $|p-q| \leq 1$ may likely be relaxed so that the same conclusion would hold in general non-Riemannian / non-Lorentzian signatures, but this would require a more careful investigation into related real forms, which is beyond our scope here.  For the more subtle conformal Riemannian and Lorentzian cases, finding the complete local classification of submaximally symmetric models is an open problem.  (See \cite{DT2014} for known models.)

 \subsection{Parabolic contact structures}
  
 Generalizing \S \ref{S:G2}, {\sl parabolic contact structures of type $(G,P)$} (or ``{\sl $G$-contact structures}'') are underlying structures for (regular, normal) geometries of types:
 \begin{align}
 \begin{split}
 &(A_\ell,P_{1,\ell}),\,\, \ell \geq 2,\quad (B_\ell,P_2),\,\, \ell \geq 3,\quad (C_\ell,P_1),\,\, \ell \geq 2,\quad (D_\ell, P_2),\,\, \ell \geq 4,\\ 
 &(E_6,P_2),\quad (E_7,P_1),\quad (E_8,P_8),\quad (F_4,P_1), \quad (G_2,P_2).
 \end{split}
 \end{align}  
 As shown in \cite{The2018}, these structures all admit descriptions (possibly passing to a correspondence space) in terms of differential equations.  The cases $(A_2, P_{1,2})$ and $(C_2,P_1)$ are classical, and correspond to scalar 2nd order ODE (up to point transformations) and scalar 3rd order ODE (up to contact transformations).  These are exceptions: they admit non-unique submaximally symmetric structures with $\fS = 3$ and $\fS = 5$ symmetries respectively.  For all other cases, explicit submaximally symmetric structures (with respect to a given $G_0$-irrep $\bbV$ of $H^2_+(\fg_-,\fg)$) were given in \cite[\S 4.2]{The2018}.  Over $\bbC$, these are locally unique by Theorem \ref{T:main}.

 \section*{Acknowledgements}

 We thank Boris Kruglikov and Henrik Winther for helpful discussions.  The research leading to these results has received funding from the Norwegian Financial Mechanism 2014-2021 (project registration number 2019/34/H/ST1/00636), the Troms\o{} Research Foundation (project ``Pure Mathematics in Norway''), the UiT Aurora project MASCOT, and this article/publication is based upon work from COST Action CaLISTA CA21109 supported by COST (European Cooperation in Science and Technology), \href{https://www.cost.eu}{https://www.cost.eu}.



\begin{thebibliography}{99}
 
 \bibitem{BLP2021} V.M.\ Boyko, O.V.\ Lokaziuk, R.O.\ Popovych, {\it Admissible transformations and Lie symmetries of linear systems of second-order ordinary differential equations}, arXiv:2105.05139 (2021).

 \bibitem{Cap2005a} A.\ \v{C}ap, {\it Automorphism groups of parabolic geometries}, Rend. Circ. Mat. Palermo (2) Suppl. {\bf 75} (2005), 233--239.

 \bibitem{Cap2005} A.\ \v{C}ap, {\it Correspondence spaces and twistor spaces for parabolic geometries}, J. Reine Angew. Math {\bf 582} (2005), 143--172.
 
 \bibitem{Cap2017} A.\ \v{C}ap, {\it On canonical Cartan connections associated to filtered G-structures}, arXiv:1707.05627 (2017).

 \bibitem{CN2009} A.\ \v{C}ap, K.\ Neusser, {\it On automorphism groups of some types of generic distributions}, Diff.\ Geom.
Appl. {\bf 27} (2009), no. 6, 769--779.
  
 \bibitem{CS2009} A. \v{C}ap, J. Slov\'ak, {\it Parabolic Geometries I: Background and General Theory}, Mathematical Surveys and Monographs, vol. 154, American Mathematical Society, 2009.
 
 \bibitem{DMT2020} B.\ Doubrov, A.\ Medvedev, and D.\ The, {\it Homogeneous integrable Legendrian contact structures
in dimension five}, J. Geom. Anal. {\bf 30} (2020), 3806--3858.
 
 \bibitem{DMT2021} B.\ Doubrov, A.\ Medvedev, and D.\ The, {\it Homogeneous Levi non-degenerate hypersurfaces
in $\mathbb{C}^3$}, Math. Z. {\bf 297} (2021), 669--709.
 
 \bibitem{DT2014} B.\ Doubrov and D.\ The, {\it Maximally degenerate Weyl tensors in Riemannian and Lorentzian signatures}, Diff. Geom. Appl. {\bf 34} (2014), 25--44.
 
 \bibitem{Egorov1951} P.\ Egorov, {\it Collineations of projectively connected spaces} (Russian), Dokl. Akad. Nauk SSSR {\bf 80} (1951), 709--712.
 
 \bibitem{Kos1961} B.\ Kostant, {\it Lie algebra cohomology and the generalized Borel--Weil theorem}, Ann. of Math. (2) {\bf 74} (1961), 329--387.

 \bibitem{Kru1954} G. I. Kruchkovich, {\it Classification of three-dimensional Riemannian spaces according to groups of motions}, Usp. Mat. Nauk (N.S.) {\bf 9}, no.1(59) (1954), 3-40; correction: Usp. Mat. Nauk (N.S.) {\bf 9}, no.3 (61) (1954), 285. (Russian) 

 \bibitem{Kru2016} B.\ Kruglikov, {\it Submaximally symmetric CR-structures}, J. Geom. Anal. {\bf 26} (2016),  3090--3097.

 \bibitem{KMT2016} B.\ Kruglikov, V.\ Matveev, D.\ The, {\it Submaximally symmetric c-projective structures}, Internat. J. Math. {\bf 27} (2016), no. 3, 1650022.

 \bibitem{KWZ2018} B.\ Kruglikov, H.\ Winther, L.\ Zalabova, {\it Submaximally symmetric almost quaternionic structures}, Transform. Groups {\bf 23} (2018), 723--741.

 \bibitem{KT2017} B.\ Kruglikov, D.\ The, {\it The gap phenomenon in parabolic geometries}, J. Reine Angew. Math. {\bf 723} (2017), 153--215.

 \bibitem{KT2018} B.\ Kruglikov, D.\ The, {\it Jet-determination of symmetries of parabolic geometries}, Math. Ann. 
 {\bf 371} (2018), 1575--1613.

 \bibitem{The2018} D.\ The, {\it Exceptionally simple PDE}, Diff. Geom. Appl. {\bf 56} (2018), 13--41.
 
 \bibitem{LiE} M.A.A. van Leeuwen, A. M. Cohen, B. Lisser, LiE, A package for Lie group computations (1992), \href{http://wwwmathlabo.univ-poitiers.fr/~maavl/LiE/}{http://wwwmathlabo.univ-poitiers.fr/$\sim$maavl/LiE/}.
 
 \bibitem{Winther2023} H.\ Winther, {\it Minimal Projective Orbits of Semi-simple Lie Groups}, arXiv:2302.12138 (2023).

 \end{thebibliography}
 \end{document}